\documentclass{article} 

\usepackage[utf8]{inputenc}
\usepackage[T1]{fontenc}
\usepackage{lmodern}
\usepackage{amsmath, amssymb, amsthm, amstext,mathrsfs}
\usepackage{microtype}
\usepackage[english=usenglishmax]{hyphsubst}
\usepackage{esint}
\usepackage{booktabs}
\usepackage{graphicx}
\usepackage[babel,german=quotes]{csquotes}
\usepackage{bibentry}
\usepackage{tabularx}		
\usepackage{commath}			
\usepackage{setspace}
\usepackage{bm}
\usepackage{threeparttable}
\usepackage{comment}\includecomment{notthisone}
\usepackage{enumerate}
\usepackage{epstopdf}
\usepackage{enumerate}
\usepackage{dsfont}
\usepackage{pstool}
\usepackage{subfig}
\usepackage{natbib}
\usepackage{hyperref}
\usepackage{authblk}
\usepackage{algorithm}
\usepackage{algcompatible}
\usepackage{algorithmicx}
\usepackage{array, booktabs}
\usepackage{setspace}
\usepackage{mathcomp}
\usepackage{multirow}
\usepackage[paper=a4paper,left=25mm,right=25mm,top=25mm,bottom=25mm]{geometry}

\floatname{algorithm}{Algorithm}

\providecommand{\keywords}[1]{\textbf{\textit{Keywords: }} #1}
\providecommand{\classification}[1]{\textit{2020 Mathematics Subject Classification: } #1}
\makeatletter
\def\BState{\State\hskip-\ALG@thistlm}
\makeatother

\usepackage{mathtools}
\mathtoolsset{showonlyrefs}

\makeatletter

\newtheoremstyle{Definition}
  {0.2cm}                   
  {0.2cm}                   
  {\normalfont}           
  {}                      
  {\bfseries}  						
  {.}                     
  { }              				
  {}
                          
\newtheoremstyle{Theorem}
  {0.2cm}                   
  {0.2cm}                   
  {\itshape}           		
  {}                      
  {\bfseries}  						
  {.}                     
  { }              				
  {}

\theoremstyle{Theorem}
	\newtheorem{cor}{Corollary}
	\newtheorem{prop}{Proposition}
	\newtheorem{lem}{Lemma}
	\newtheorem{thm}{Theorem}

	\newtheorem{ass}{Assumption}

\theoremstyle{Definition}
	\newtheorem{example}{Example}
	\newtheorem{rem}{Remark}

\newcommand{\Levy}{L\'{e}vy }
\newcommand{\upd}{\mathrm{d}}

\allowdisplaybreaks
\makeatother

\begin{document}




\title{Approximation and Error Analysis of Forward-Backward SDEs driven by General L\'evy Processes using Shot Noise Series Representations} 
\author[a]{Till Massing\thanks{Faculty of Economics, University of
Duisburg-Essen, Universit{\"{a}}tsstr.~12, 45117 Essen, Germany.\\E-Mail:
till.massing@uni-due.de}}

\maketitle

\begin{abstract}
We consider the simulation of a system of decoupled forward-backward stochastic differential equations (FBSDEs) driven by a pure jump L\'evy process $L$ and an independent Brownian motion $B$. We allow the L\'evy process $L$ to have an infinite jump activity. Therefore, it is necessary for the simulation to employ a finite approximation of its L\'evy measure. We use the generalized shot noise series representation method by \cite{Rosinski2001} to approximate the driving L\'evy process $L$. We compute the $L^p$ error, $p\ge2$, between the true and the approximated FBSDEs which arises from the finite truncation of the shot noise series (given sufficient conditions for existence and uniqueness of the FBSDE). We also derive the $L^p$ error between the true solution and the discretization of the approximated FBSDE using an appropriate backward Euler scheme.
\end{abstract}

\keywords{Decoupled forward-backward SDEs with jumps; L\'evy processes; Shot noise series representation; Discrete-time approximation; Euler Scheme}\\
\classification{60H10; 60H35; 65C05}
\microtypesetup{activate=true}

\section{Introduction}\label{sec:intro}

We consider a system of decoupled forward-backward stochastic differential equations (FBSDE) with jumps of the type
\begin{align}
X_t&=X_0+\int_0^tb (s,X_s)\upd s+\int_0^ta(s,X_s)\upd B_s+\int_0^t\int_{\mathbb{R}}h(s,X_{s-})e\widetilde{\mu}(\upd e,\upd s)\label{eq:XSDEv3}\\
Y_t&= g(X_T)+\int_t^Tf (s,X_s,Y_s,Z_s,\Gamma_s)\upd s-\int_t^TZ_s\upd B_s-\int_t^T\int_{\mathbb{R}}U_s(e)e\widetilde{\mu}(\upd e,\upd s)\label{eq:YSDEv3},
\end{align}
for $t\in[0,T]$, where $B$ is a Brownian motion and $L$ is an independent pure jump \Levy process with \Levy measure $\nu$ where $\mu(\upd e,\upd t)=\nu(\upd e)\upd t$ is the corresponding Poisson measure and $\widetilde{\mu}(\upd e,\upd s)=\mu(\upd e,\upd s)-\nu(\upd e)\upd s$ the compensated Poisson measure. Let $\Gamma_s=\int_{\mathbb{R}}\rho(e)U_s(e)e\nu(\upd e)$. We discuss the case where $L$ has an infinite jump activity, i.e., $\nu(\mathbb{R})=\infty$. (Assumptions on the functions $a,b,f,g,h$ are discussed below.) For path simulation it is therefore necessary to employ a finite approximation of the \Levy measure first. Afterwards we have to discretize the FBSDE. We are interested in path simulation and the associated error between the true solution and the approximate solution of \eqref{eq:XSDEv3}-\eqref{eq:YSDEv3}.

Backward SDEs are a vibrant research topic since the seminal paper of \cite{Pardoux1990}. They proved existence and uniqueness in the $L^2$ sense of a solution of a BSDE (without jumps) under the assumptions of square integrability of the terminal condition and Lipschitz continuity of the generator $f$. Since then BSDEs and/or FBSDEs have been analyzed in many directions.

One strand of the literature treats extensions of the existence and uniqueness result of \cite{Pardoux1990} by relaxing the underlying assumptions or extending the BSDE under consideration. For example, \cite{Tang1994} and \cite{Barles1997} included jumps into the BSDE. \cite{Briand2003} discussed the existence and uniqueness in an $L^p$ sense given a Brownian filtration. \cite{Buckdahn1994} did the same including jumps. Since then several papers have shown $L^p$ existence and uniqueness with a generalized filtration under weak assumptions, e.g., \cite{Kruse2016}, \cite{Yao2017}, and \cite{Eddahbi2017}.
FBSDEs are the Markovian special case of BSDEs where the terminal condition is determined by the forward SDE. 

Another strand of the literature covers possible areas of applications of FBSDEs. For example, FBSDEs turned out to be useful in mathematical finance, see \cite{Karoui1997} and \cite{Delong2013}, in optimal control, see \cite{Tang1994}, or for partial differential equations, see \cite{Pardoux1999} or also the book \cite{Pardoux2014}.

A third strand is about the discrete-time approximation of FBSDEs, which this paper aims to contribute to. 
A popular approach is a backward Euler scheme, see \cite{Zhang2004} and \cite{Bouchard2004} who derived the $L^2$ approximation error of the scheme. \cite{Gobet2007} generalized this by computing the $L^p$ error. \cite{Bouchard2008} derived the $L^2$ error for FBSDEs containing a finite number of jumps. In the case of an infinite jump activity, \cite{Aazizi2013} proposed a two-step approximation by first approximating the small jumps by a Brownian motion to have only finitely many big jumps, and second by discretizing according to \cite{Bouchard2008}. \cite{Aazizi2013} then derived the $L^2$ approximation-discretization error. The approach follows \cite{Kohatsu-Higa2010} who approximated forward SDEs with infinitely many jumps by finitely many jumps.

This paper contributes to literature in the following way. First, we extend the results of \cite{Aazizi2013} for the $L^2$ approximation-discretization error to a more general $L^p$, $p\ge2$, version by incorporating a jump-adapted Euler scheme. Second, instead of partitioning the \Levy measure into jumps larger or smaller than a certain level we allow for various truncation functions using the approach of shot noise series representations by \cite{Rosinski2001}, which may be the more efficient way for a certain \Levy process (some examples are provided below). Third, we enlarge the class of pure jump \Levy processes to these which do not fulfill the \cite{Asmussen2001} assumption for the approximation of small jumps. All in all, we obtain a statement for the $L^p$ error for general \Levy processes. We find that the error depends on $N^{-1/2}$, where $N$ is the number of time steps, and on the $p$th and second moments of the \Levy measure of the discarded jumps.

The remainder of this paper is organized as follows. In Section \ref{sec:settings} we discuss the settings in more detail. In Section \ref{sec:approx} we derive an upper bound for the error of the approximation with a finite jump measure. In Section \ref{sec:discrete} we present the discrete Euler scheme and prove an upper bound for the discretization error.

\section{Settings}\label{sec:settings}

This section introduces the setting and notation needed throughout this paper. Let $(\Omega,\mathcal{F},(\mathcal{F}_t)_{0\le t\le T},\mathbb{P})$ be a filtered probability space such that $\mathcal{F}_0$ contains the $\mathbb{P}$-null sets, $\mathcal{F}_T=\mathcal{F}$, and $(\mathcal{F}_t)$ satisfies the usual assumptions. We assume that $(\mathcal{F}_t)$ is generated by a one-dimensional Brownian motion $B$ and an independent Poisson measure $\mu$ on $[0,T]\times\mathbb{R}$ with intensity $\nu(\upd e)\upd t$, where $\nu$ is a \Levy measure on $\mathbb{R}$ and $\mathbb{R}$ is equipped with the Borel set $\mathcal{B}:=\mathcal{B}(\mathbb{R})$.
We assume that the \Levy measure $\nu$ satisfies $\int_{\mathbb{R}}(1\wedge |e|^2)\nu(\upd e)\le K<\infty$, for a constant $K>0$ (the Lipschitz constant from below) and that $\nu(\mathbb{R})=\infty$. If $\nu(\mathbb{R})<\infty$, a finite jump approximation would not be necessary and we could skip Section \ref{sec:approx}.

Furthermore, we assume that
\begin{equation}
\label{eq:measureass}
\int_{\mathbb{R}}|e|^p\nu(\upd e)<\infty,
\end{equation}
for $p\ge2$. This implies that the $p$th moment of $L_t$ for each $t\in[0,T]$ is finite.
We denote by $\widetilde{\mu}(\upd e,\upd s)=\mu(\upd e,\upd s)-\nu(\upd e)\upd s$ the compensated Poisson measure corresponding to $\mu$.

For $p\ge2$ we define the normed spaces on $[r,t]$, $r\le t$,
\begin{itemize}
\item $\mathcal{S}_{[r,t]}^p$ is the set of real-valued adapted càdlàg processes $Y$ such that
\begin{equation}
\label{eq:Ynorm}
||Y||_{\mathcal{S}_{[r,t]}^p}:= \mathbb{E}\left[\sup_{r\le s\le t}|Y_s|^p\right]^{1/p}<\infty.
\end{equation}
\item $\mathbb{H}_{[r,t]}^p$ is the set of progressively measurable $\mathbb{R}$-valued processes $Z$ such that
\begin{equation}
\label{eq:Znorm}
||Z||_{\mathbb{H}_{[r,t]}^p}:= \mathbb{E}\left[\left(\int_r^t|Z_s|^2\upd s\right)^{p/2}\right]^{1/p}<\infty.
\end{equation}
\item $\mathbb{L}_{\mu,[r,t]}^p$ is the set of $(\mathcal{P}\otimes \mathcal{B})$-measurable maps $U:\Omega\times[0,T]\times \mathbb{R}\rightarrow\mathbb{R}$ such that
\begin{equation}
\label{eq:Unorm}
||U||_{\mathbb{L}_{\mu,[r,t]}^p}:=\mathbb{E}\left[\left(\int_r^t\int_{\mathbb{R}}|U_s(e)e|^2\nu(\upd e)\upd s\right)^{p/2}\right]^{1/p}<\infty,
\end{equation}
where $\mathcal{P}$ is the $\sigma$-algebra of $(\mathcal{F}_t)$-predictable subsets of $\Omega\times[0,T]$.
\item $\mathbb{L}_{\nu}^p$ is the set of measurable maps $U: \mathbb{R}\rightarrow\mathbb{R}$ such that
\begin{equation}
\label{eq:Unormnu}
||U||_{\mathbb{L}_{\nu}^p}:=\left(\int_{\mathbb{R}}|U(e)e|^p\nu(\upd e)\right)^{1/p}<\infty.
\end{equation}
\item The space $\mathcal{E}_{[r,t]}^p:=\mathcal{S}_{[r,t]}^p\times\mathbb{H}_{[r,t]}^p\times\mathbb{L}_{\mu,[r,t]}^p$ is endowed with the norm
\begin{equation}
\label{eq:allnorm}
||(Y,Z,U)||_{\mathcal{E}_{[r,t]}^p}:=\left(||Y||_{\mathcal{S}_{[r,t]}^p}^p+||Z||_{\mathbb{H}_{[r,t]}^p}^p+||U||_{\mathbb{L}_{\mu,[r,t]}^p}^p\right)^{1/p}.
\end{equation}
\end{itemize}
In the remainder we omit the subscripts if $[r,t]=[0,T]$, e.g., $\mathcal{E}^p:=\mathcal{E}_{[0,T]}^p$.

We introduce the set of assumptions needed throughout the proofs. Note that these assumptions are not the minimal ones needed for existence and uniqueness. However they are not overly restrictive and used frequently throughout the literature, e.g., \cite{Bouchard2008} and \cite{Aazizi2013}.
\begin{ass}\label{ass:Lipschitz}
\begin{enumerate}
\item[(i)] Let $a:\mathbb{R}\times\mathbb{R}\to\mathbb{R}$, $b:\mathbb{R}\times\mathbb{R}\to\mathbb{R}$, $h:\mathbb{R}\times\mathbb{R}\to\mathbb{R}$ be Lipschitz continuous functions w.r.t.~$x$ and $\frac{1}{2}$-Hölder continuous w.r.t.~$t$, i.e., for a constant $K>0$
\begin{equation}
\label{eq:LipschitzX}
|b(t,x)-b(t',x')|+|a(t,x)-a(t',x')|+|h(t,x)-h(t',x')|+|g(t,x)-g(t',x')|\le K \left(|t-t'|^{1/2}+|x-x'|\right)
\end{equation}
is satisfied for all $(t,x),(t',x')\in[0,T]\times\mathbb{R}$.
\item[(ii)] Let $f:[0,T]\times\mathbb{R}\times\mathbb{R}\times\mathbb{R}\times\mathbb{R}\to\mathbb{R}$ such that it is Lipschitz continuous w.r.t.~$(x,y,z,q)$ and $\frac{1}{2}$-Hölder continuous w.r.t.~$t$, i.e., for a constant $K>0$
\begin{equation}
\label{eq:LipschitzY}
|f(t,x,y,z,q)-f(t,x',y',z',q')|\le K(|t-t'|^{1/2}+|x-x'|+|y-y'|+|z-z'|+|q-q'|)
\end{equation}
is satisfied for all $(t,x,y,z,q),(t',x',y',z',q')\in[0,T]\times\mathbb{R}\times\mathbb{R}\times\mathbb{R}\times\mathbb{R}$.
\item[(iii)]
Let $\rho:\mathbb{R}\to\mathbb{R}$ be a measurable function such that for a constant $K>0$
\begin{equation}\label{eq:boundonrho}
|\rho(e)|\le K(1\wedge|e|),
\end{equation}
for all $e\in\mathbb{R}$.
\item[(iv)] For $p\ge2$ the integrability condition
\begin{equation}
\label{eq:assumptiononf}
\mathbb{E}\left[|g(X_T)|^p+\int_0^T|f(t,0,0,0,0)|^p\upd t \right]<\infty.
\end{equation}
is satisfied.
\end{enumerate}
\end{ass}

To prove Theorem \ref{thm:EulerErr} we need the following additional assumption. A discussion about it can be found in Remark \ref{rem:Ass2}.
\begin{ass}\label{ass:invofh}
The function $h(t,x)$ is differentiable in the $x$ variable with derivative $h_x(t, x)$ such that for each $e\in\mathbb{R}$ the function
$\ell(t,x;e):=h_x(t,x)e+1$
satisfies one of the following conditions
\begin{equation}
\ell(t,x;e)\ge K^{-1}\ \ \ \mathrm{or}\ \ \ \ell(t,x;e)\le-K^{-1}.
\end{equation}
uniformly in $(t,x)\in[0,T]\times\mathbb{R}$.
\end{ass}
We choose a constant $K$ that fulfills both assumptions.
We next mention some important facts on \Levy processes which we need to approximate the infinite \Levy measure. We opt for the approximation using series representations which goes back to \cite{Rosinski2001}, see \cite{Yuan2021} for a recent overview. Let $L$ be a pure jump \Levy process on $(\Omega,\mathcal{F},(\mathcal{F}_t)_{0\le t\le T},\mathbb{P})$ with \Levy measure $\nu$ as discussed above. The famous L\'evy-Itô decomposition states that $L$ can be written as
\begin{equation}
\label{eq:LevyItorewritten2}
L_t=\int_0^t\int_{|e|\le1}e\widetilde{\mu}(\upd e,\upd s)+\int_0^t\int_{|e|>1}e\mu(\upd e,\upd s)
\end{equation}
for $t\in[0,T]$.

\cite{Rosinski2001} proved the useful result that it is possible to express jump-type \Levy processes as an infinite series. We now summarize his theory of \emph{generalized shot noise series representations}. We only present the one-dimensional case (it is of course also available in $d$ dimensions). Suppose that the \Levy measure $\nu$ can be decomposed as
\begin{equation}
\nu(B)=\int_0^{\infty}\mathbb{P}[H(r,V)\in B]\upd r,\ \ \ B\in\mathcal{B},
\end{equation}
where $V$ is a random variable in some space $\mathcal{V}$ and $H:(0,\infty)\times\mathcal{V}\to\mathbb{R}$ is a measurable function such that for every $v\in\mathcal{V}$, $r\mapsto |H(r,v)|$ is nonincreasing. Then, it holds that
\begin{equation}\label{eq:shotnoise}
L_t\stackrel{\mathcal{L}}{=}\sum_{i=1}^{\infty}H\left(\frac{G_i}{T},V_i\right)\mathds{1}_{[0,t]}(T_i)-tc_i,
\end{equation}
for $t\in[0,T]$, where $\{G_i\}_{i\in\mathbb{N}}$ are the arrival times of a standard Poisson process (i.e., a sequence of partial sums of i.i.d.~unit rate exponential random variables), $\{V_i\}_{i\in\mathbb{N}}$ are i.i.d.~copies of the random variable $V$ independent of $\{G_i\}_{i\in\mathbb{N}}$, $\{T_i\}_{i\in\mathbb{N}}$ are i.i.d.~uniforms on $[0,T]$ independent of $\{G_i\}_{i\in\mathbb{N}}$ and $\{V_i\}_{i\in\mathbb{N}}$, and $\{c_k\}_{k\in\mathbb{N}}$ are centering constants such that
\begin{equation}
c_i=\int_{i-1}^i\int_{|x|\le 1}x\mathbb{P}[H(r,V)\in\upd x]\upd r.
\end{equation}
Rosi\'nski's theorem offers several choices for different series representations. The most convenient representation is case-dependent given the specific \Levy measure. Well-known special cases include the inverse \Levy measure method, the rejection method or the thinning method, see \cite{Yuan2021} for details.

To obtain a feasible numerical algorithm one has to truncate the infinite series in \eqref{eq:shotnoise}. Instead of truncating the series deterministically, i.e., after $n$ summands, we choose a random truncation
\begin{equation}\label{eq:truncshotnoise}
L_t^{n}:=\sum_{\{i:G_i\le nT\}}H\left(\frac{G_i}{T},V_i\right)\mathds{1}_{[0,t]}(T_i)-tc_i,
\end{equation}
where we cut off all summands if $G_i> nT$, which depends on the random Poisson arrival times $\{G_i\}$. The reason is that with the random truncation $L^n$ itself is a compound Poisson process and hence a proper \Levy process with \Levy measure
\begin{equation}
\label{eq:truncmeasure}
\nu^n(B)=\int_0^n\mathbb{P}[H(r,U)\in B]\upd r,\ \ \ B\in \mathcal{B}(\mathbb{R}^d).
\end{equation}
Note that the truncated \Levy measure only describes finitely many jumps, i.e., $\nu^n(\mathbb{R})=n<\infty$. In the following we use the notation $\bar{\nu}^n(\upd e):=\nu(\upd e)-\nu^n(\upd e)$, which is the \Levy measure of the infinitely many small jumps that are discarded.
The important quantity which determines the approximation error of FBSDEs will be in terms of the second and $p$th moments of the \Levy measure $\bar{\nu}^n$, which are defined as
\begin{equation}
\label{eq:sigmaerror}
\sigma^p(n):=\int_{\mathbb{R}}|e|^p\bar{\nu}^n(\upd e),
\end{equation}
for $p\ge2$. 
By $\mu^n$ we denote the Poisson measure with intensity measure $\nu^n(\upd e)\upd t$ and by $\bar{\mu}^n$ the Poisson measure with intensity measure $\bar{\nu}^n(\upd e)\upd t$. Let $\widetilde{\mu^n}$ and $\widetilde{\bar{\nu}^n}$ be the corresponding compensated Poisson measures. Clearly, $\nu=\nu^n+\bar{\nu}^n$, $\mu=\mu^n+\bar{\mu}^n$ and $\widetilde{\mu}=\widetilde{\mu^n}+\widetilde{\bar{\mu}^n}$.

We give some examples for special \Levy processes with different series representations.
\begin{example}
We first consider the case where $L$ is a gamma process, i.e., a \Levy process with \Levy measure $\alpha \frac{\exp(-\beta e)}{e}\mathds{1}(e>0)\upd e$, with $\alpha>0,\beta>0$. There are four widely used series representation for the gamma process \cite[]{Rosinski2001}.
\begin{align}
L_t&\stackrel{\mathcal{L}}{=}\frac{1}{\beta}\sum_{i=1}^{\infty}E_1^{-1}\left(\frac{G_i}{\alpha T}\right)\mathds{1}_{[0,t]}(T_i),&\textrm{(inverse  \Levy  measure  method)}\\
L_t&\stackrel{\mathcal{L}}{=}\frac{1}{\beta}\sum_{i=1}^{\infty}\frac{1}{\mathrm{e}^{G_i/(\alpha T)}-1}\mathds{1}\left(\frac{\mathrm{e}^{G_i/(\alpha T)}}{\mathrm{e}^{G_i/(\alpha T)}-1}\mathrm{e}^{-(\mathrm{e}^{G_i/(\alpha T)})^{-1}}\ge V_i^{(1)}\right)\mathds{1}_{[0,t]}(T_i),&\textrm{(rejection method)}\\
L_t&\stackrel{\mathcal{L}}{=}\frac{1}{\beta}\sum_{i=1}^{\infty}V_i^{(2)}\mathds{1}\left(V_i^{(2)}G_i\le \alpha T\right)\mathds{1}_{[0,t]}(T_i),&\textrm{(thinning method)}\\
L_t&\stackrel{\mathcal{L}}{=}\frac{1}{\beta}\sum_{i=1}^{\infty}\mathrm{e}^{-\frac{G_i}{\alpha T}}V_i^{(2)}\mathds{1}_{[0,t]}(T_i),&\textrm{(Bondesson method)}
\end{align}
for $t\in[0,T]$, where $\{G_i\}_{i\in\mathbb{N}}$ are the arrival times of a standard Poisson process, $\{V_i^{(1)}\}_{i\in\mathbb{N}}$ are i.i.d.~uniforms on $[0,1]$, $\{V_i^{(2)}\}_{i\in\mathbb{N}}$ are i.i.d.~standard exponential r.v.'s, $\{T_i\}_{i\in\mathbb{N}}$ are i.i.d.~uniforms on $[0,T]$. All sequences are independent of each other. $E_1(x):=\int_u^{\infty}\frac{\mathrm{e}^{-u}}{u}\upd u$ is the exponential integral and $E_1^{-1}$ its inverse. The centering constants $c_i$ are all equal to zero in this case. \cite{Imai:2013:NIL:2489318.2489552} showed the superiority of the inverse \Levy measure method compared with the other methods when employing a finite truncation, i.e., its $\sigma^p(n)$ is smaller for $p\ge2$. However, in practice it may be easier to work with the other methods. For example, for the gamma process the inverse \Levy measure method relies on numerical inversion of the exponential integral function. The method of Bondesson may be the easiest to use in this case. Note that the gamma process does not fulfill the \cite{Asmussen2001} assumption for the approximation of small jumps.
\end{example}

\begin{example}
The second example are tempered stable \Levy processes, see \cite{Rosinski2007}. For simplicity, we here consider the classical tempered stable subordinator $L$ with \Levy measure $\delta \frac{\exp(-\lambda e)}{e^{-1-\alpha}}\mathds{1}(e>0)\upd e$, with $\alpha\in(0,1),\delta>0,\lambda>0$. Again, the inverse \Levy measure method, the rejection method, and the thinning method can be derived, see \cite{Kawai2011}. We omit these here and instead only present the method derived by \cite{Rosinski2007} which works best in practice. As in the example above, the inverse \Levy measure method is theoretically superior but numerically challenging. The series representation of \cite{Rosinski2007} is
\begin{equation}
L_t\stackrel{\mathcal{L}}{=}\sum_{i=1}^{\infty}\left(\left(\frac{\alpha G_i}{\delta T}\right)^{-1/\alpha}\wedge\frac{V_i U_i^{1/\alpha}}{\lambda}\right)\mathds{1}_{[0,t]}(T_i)
\end{equation}
for $t\in[0,T]$, where $\{G_i\}_{i\in\mathbb{N}}$ are the arrival times of a standard Poisson process, $\{U_i\}_{i\in\mathbb{N}}$ are i.i.d.~uniforms on $[0,1]$, $\{V_i\}_{i\in\mathbb{N}}$ are i.i.d.~standard exponential r.v.'s, $\{T_i\}_{i\in\mathbb{N}}$ are i.i.d.~uniforms on $[0,T]$. All sequences are independent of each other. The centering constants $c_i$ are all equal to zero in this case.
\end{example}

\begin{rem}\label{rem:boundonrho}
The \Levy measures $\nu$ and $\bar{\nu}^n$ are assumed to be infinite, i.e., $\nu(\mathbb{R})=\bar{\nu}^n(\mathbb{R})=\infty$ and thus we cannot apply Jensen's inequality to integrals of the type $\int \nu(\upd e)$. Assumption \ref{ass:Lipschitz}.(iii) provides a necessary bound. Indeed, by Hölder's inequality
\begin{equation}
\left(\int_{\mathbb{R}}\rho(e) U_s(e)e\nu(\upd e)\right)^2\le \int_{\mathbb{R}} U_s(e)^2e^2\nu(\upd e)\int_{\mathbb{R}}\rho(e)^2\nu(\upd e)\le K^3 \int_{\mathbb{R}} U_s(e)^2e^2\nu(\upd e)
\end{equation}
for $U_s\in\mathbb{L}_{\mu}^2$, because $\int(1\wedge|e|^2)\nu(\upd e)\le K<\infty$ is bounded by a finite constant.
\end{rem}


Given the approximation of \Levy processes using truncated series representations we use $\widetilde{\mu^n}$ to approximate the Poisson measure $\widetilde{\mu}$. We call $(X^n,Y^n,Z^n,U^n)$ the solution of the approximate FBSDE
\begin{align}
X_t^{n}&=X_0+\int_0^tb (s,X_s^{n})\upd s+\int_0^ta(s,X_s^{n})\upd B_s+\int_0^t\int_{\mathbb{R}}h(s,X_{s-}^{n})e\widetilde{\mu^n}(\upd e,\upd s)\label{eq:nXSDEv3}\\
Y_t^{n}&= g(X_T^{n})+\int_t^Tf (s,X_s^{n},Y_s^{n},Z_s^n,\Gamma_s^{n})\upd s-\int_t^TZ_s^{n}\upd B_s-\int_t^T\int_{\mathbb{R}}U_s^{n}(e)e\widetilde{\mu^n}(\upd e,\upd s),\label{eq:nYSDEv3}
\end{align}
where
$\Gamma_s^{n}=\int_{\mathbb{R}}\rho(e)U_s^{n}(e)e\nu^{n}(\upd e)$. The aim of the next section is to compute the approximation error between the original FBSDE \eqref{eq:XSDEv3}-\eqref{eq:YSDEv3} and the approximate FBSDE \eqref{eq:nXSDEv3}-\eqref{eq:nYSDEv3}.

\begin{rem}
In this paper we restrict ourselves to one-dimensional FBSDEs. The extension to multidimensional FBSDEs (the comparison principle, see \cite{Barles1997}, only holds for one-dimensional BSDEs) is straightforward. We can replace Itô's formula by its multidimensional counterpart in the proofs and use the same arguments and bounds to obtain the multidimensional results. We omit these details in the proofs in favor of a simpler notation.
\end{rem}

 We end this section with a notational remark. Let $C_p$ denote a generic constant depending only on $p$ and further constants including $K$, $T$, $a(0)$, $b(0)$, $f(0)$, $g(0)$, $h(0)$ and the starting value $X_0$, which may vary from step to step.

%
%

\section{Error of the approximation of the pure jump process}\label{sec:approx}

In this section we compute the $L^p$ approximation error between the original backward SDE \eqref{eq:YSDEv3} and the approximate backward SDE \eqref{eq:nYSDEv3}, defined as
\begin{align}
Err_{n,p}(Y,Z,U):= &\ \mathbb{E}\left[\sup_{0\le t\le T}|Y_t-Y_t^{n}|^p+\left(\int_0^T|Z_s-Z_s^{n}|^2\upd s\right)^{p/2}\right.\\
&+\left.\left(\int_0^T\int_{\mathbb{R}}|U_s(e)-U_s^{n}(e)|^2e^2\nu^n(\upd e)\upd s\right)^{p/2}+\left(\int_0^T\int_{\mathbb{R}}U_s(e)^2e^2\bar{\nu}^n(\upd e)\upd s\right)^{p/2}\right]^{1/p}.
\end{align}
Furthermore, we derive an upper bound for the approximation error of the forward SDE defined as
\begin{equation}
\mathbb{E}\left[\sup_{0\le t\le T}|X_t-X_t^{n}|^p\right].
\end{equation}

In the following proofs, the standard estimate for the solution of the FBSDEs is useful:
\begin{equation}\label{eq:standardestimate}
||(X,Y,Z,U)||_{\mathcal{S}^p\times\mathcal{E}^p}^p\le C_p(1+|X_0|^p),
\end{equation}
for $p\ge2$, see \cite{Bouchard2008}. In particular, the forward SDE has the estimate
\begin{equation}\label{eq:standardestimateX}
\mathbb{E}\left[\sup_{0\le t\le T}|X_t|^p\right]\le C_p(1+|X_0|^p).
\end{equation}

Because the FBSDEs are decoupled we can analyze the forward and backward components separately. We begin with an error bound for the forward SDE.
\begin{prop}\label{prop:approxX}
Let $p\ge2$. Under Assumption \ref{ass:Lipschitz} on $(\Omega,\mathcal{F},(\mathcal{F}_t),\mathbb{P})$
\begin{itemize}
\item there exists a unique solution $X$ of \eqref{eq:XSDEv3} on $[0,T]$ with $X_0=0$,
\item for any $n\in\mathbb{N}$, there exists a unique solution $X^{n}$ of \eqref{eq:nXSDEv3} on $[0,T]$ with $X_0^{n}=0$,
\end{itemize}
Moreover, there exists a constant $C_p$ such that
\begin{equation}
\label{eq:approxerrorX}
\mathbb{E}\left[\sup_{0\le t\le T}|X_t-X_t^{n}|^p\right]\le C_p \left(\sigma^p(n)+\sigma^2(n)^{p/2}\right).
\end{equation}
\end{prop}

\begin{proof}
The existence and uniqueness is a standard result, see \cite{Applebaum2009}. We thus only prove the bound \eqref{eq:approxerrorX} which is an easy extension of \cite{Aazizi2013}. 

Let $t\le T$. We plug in the SDEs and use the Burkholder-Davis-Gundy (see Theorem 2.11 of \cite{Kunita2004}) inequality to obtain
\begin{align}
\mathbb{E}\left[\sup_{0\le s\le t}|X_s-X_s^{n}|^p\right]  \le C_p &\left( \mathbb{E}\left[\left(\int_0^t |b (r,X_r)-b (r,X_r^{n})|\upd r\right)^p\right]\right.\\
&+\mathbb{E}\left[\left(\int_0^t|a(r,X_r)-a(r,X_r^{n})|^2\upd r\right)^{p/2}\right]\\
&+\mathbb{E}\left[\left(\int_0^t\int_{\mathbb{R}}|h(r,X_r)-h(r,X_r^{n})|^2|e|^2\nu^n(\upd e)\upd r\right)^{p/2}\right]\\
&+\mathbb{E}\left[\int_0^t\int_{\mathbb{R}}|h(r,X_r)-h(r,X_r^{n})|^p|e|^p\nu^n(\upd e)\upd r\right]\\
&+\mathbb{E}\left[\left(\int_0^t\int_{\mathbb{R}}|h(r,X_r)|^2|e|^2\bar{\nu}^n(\upd e)\upd r\right)^{p/2}\right]\\
&\left.+\mathbb{E}\left[\int_0^t\int_{\mathbb{R}}|h(r,X_r)|^p|e|^p\bar{\nu}^n(\upd e)\upd r\right]\right)\label{eq:sideproof1}.
\end{align}

By Jensen's inequality for the $\upd s$ integral, the Lipschitz assumption \eqref{eq:LipschitzX} and \eqref{eq:standardestimate},
\begin{align}
\mathbb{E}\left[\sup_{0\le s\le t}|X_s-X_s^{n}|^p\right] & \le C_p \left( \mathbb{E}\left[\int_0^t |X_r-X_r^{n}|^p\upd r\right]+\mathbb{E}\left[\left(\int_0^t\int_{\mathbb{R}}(1+|X_r|^2)|e|^2\bar{\nu}^n(\upd e)\upd r\right)^{p/2}\right]\right.\\
&\qquad \qquad \left.+\mathbb{E}\left[\int_0^t\int_{\mathbb{R}}(1+|X_r|^p)|e|^p\bar{\nu}^n(\upd e)\upd r\right]\right)\\
&\le C_p\left(\int_0^t\mathbb{E}\left[\sup_{0\le s\le r}|X_s-X_s^{n}|^p\right]\upd r + \sigma^2(n)^{p/2}+\sigma^p(n)\right).
\end{align}
Now the result follows from Gronwall's lemma.
\end{proof}

We now turn to the approximation error of the backward SDE. We start with a remark which gives some insights into the error and next we state and prove our first main result. 
\begin{rem}
Observe that, by the Burkholder-Davis-Gundy inequality,
\begin{align}
&\mathbb{E}\left[\sup_{0\le t\le T}\left|\int_t^T\int_{\mathbb{R}}U_s(e)e\widetilde{\mu}(\upd e,\upd s)-\int_t^T\int_{\mathbb{R}}U_s^{n}(e)e\widetilde{\mu^n}(\upd e,\upd s)\right|^p\right]\\
\le&\ C_p \mathbb{E}\left[\sup_{0\le t\le T}\left|\int_t^T\int_{\mathbb{R}}(U_s(e)-U_s^{n}(e))e\widetilde{\mu^n}(\upd e,\upd s)\right|^p+\sup_{0\le t\le T}\left|\int_t^T\int_{\mathbb{R}}U_s(e)e\widetilde{\bar{\mu}^n}(\upd e,\upd s)\right|^p\right]\\
\le&\ C_p \mathbb{E}\left[\left(\int_0^T\int_{\mathbb{R}}|U_s(e)-U_s^{n}(e)|^2e^2\nu^n(\upd e)\upd s\right)^{p/2}+\left(\int_0^T\int_{\mathbb{R}}U_s(e)^2e^2\bar{\nu}^n(\upd e)\upd s\right)^{p/2}\right]
\end{align}
\end{rem}

\begin{thm}\label{thm:approxY}
Let $p\ge2$. Under Assumption \eqref{ass:Lipschitz} on $(\Omega,\mathcal{F},(\mathcal{F}_t),\mathbb{P})$
\begin{itemize}
\item there exists a unique solution $(Y,Z,U)$ of \eqref{eq:YSDEv3} in $\mathcal{E}^p$,
\item for any $n\in\mathbb{N}$, there exists a unique solution $(Y^{n},Z^{n},U^{n})$ of \eqref{eq:nYSDEv3} in $\mathcal{E}^p$.
\end{itemize}
Moreover, 
there exists a constant $C_p$ such that
\begin{equation}
\label{eq:approxerrorY}
Err_{n,p}(Y,Z,U)^p 
\le C_p\left(\sigma^p(n)+\sigma^2(n)^{p/2}\right).
\end{equation}
\end{thm}

\begin{proof}
We omit the proof of existence and uniqueness because it can be found in the vast literature. The standard way is to first show existence and uniqueness in the space $\mathcal{E}^2$ and second show that the solution also belongs to $\mathcal{E}^p$. We refer to \cite{Barles1997}, \cite{Briand2003}, \cite{Buckdahn1994}, \cite{Kruse2016} and \cite{Eddahbi2017}. Some of their techniques also provide to be useful for the derivation of the error.

Define
\begin{align}
\delta Y_t :=&\ Y_t-Y_t^{n}\\
=&\ g(X_T)-g(X_T^{n})+\int_t^Tf (\Theta_s)\upd s-\int_t^Tf (\Theta_s^{n})\upd s-\left(\int_t^TZ_s\upd B_s-\int_t^TZ_s^{n}\upd B_s\right)\\
&-\left(\int_t^T\int_{\mathbb{R}}U_s(e)e\widetilde{\mu}(\upd e,\upd s)-\int_t^T\int_{\mathbb{R}}U_s^{n}(e)e\widetilde{\mu^n}(\upd e,\upd s)\right)\\
=&\ \delta g(X_T) + \int_t^T\delta f (\Theta_s)\upd s-\int_t^T\delta Z_s\upd B_s\\
&-\left(\int_t^T\int_{\mathbb{R}}U_s(e)e\widetilde{\bar{\mu}^n}(\upd e,\upd s)+\int_t^T\int_{\mathbb{R}}\delta U_s(e)e\widetilde{\mu^n}(\upd e,\upd s)\right),
\end{align}
where we use the notations $\delta g(X_T):=g(X_T)-g(X_T^{n})$, $\delta U_s(e):=U_s(e)-U_s^{n}(e)$, $\delta f (\Theta_s):= f (\Theta_s)-f (\Theta_s^{n})$ and $\delta\Theta_s:=(s,\delta X_s, \delta Y_s, \delta Z_s, \delta \Gamma_s):=(s,X_s-X_s^{n},Y_s-Y_s^{n},Z_s-Z_s^{n},\Gamma_s-\Gamma_s^{n})$.

\noindent \textbf{Step 1:}
We apply the Itô formula with the $C^2$-function $\eta(y)=|y|^p$ to the process $\delta Y_t$. We use that 
\begin{equation}
\frac{\partial\eta}{\partial y}(y)=py|y|^{p-2},\ \ \ \frac{\partial^2\eta}{\partial y^2}(y)=p|y|^{p-2}+p(p-2)|y|^{p-2}=p(p-1)|y|^{p-2}.
\end{equation}
Hence
\begin{align}
|\delta Y_t|^p=&\ |\delta g(X_T)|^p+ \int_t^T p\delta Y_s|\delta Y_s|^{p-2}\delta f (\Theta_s)\upd s\\
&-p\int_t^T\delta Y_{s-}|\delta Y_{s-}|^{p-2}\delta Z_s\upd B_s-\frac{1}{2}\int_t^Tp(p-1)|\delta Y_s|^{p-2}\delta Z_s^2\upd s\\
&-\int_t^T\int_{\mathbb{R}}\left(|\delta Y_{s-}+U_s(e)e|^p-|\delta Y_{s-}|^p-p\delta Y_{s-}|\delta Y_{s-}|^{p-2}U_s(e)e\right)\bar{\mu}^n(\upd e,\upd s)\\
&-\int_t^T\int_{\mathbb{R}}\left(|\delta Y_{s-}+\delta U_s(e)e|^p-|\delta Y_{s-}|^p-p\delta Y_{s-}|\delta Y_{s-}|^{p-2}\delta U_s(e)e\right)\mu^n(\upd e,\upd s)\\
&-p\int_t^T\int_{\mathbb{R}}\delta Y_{s-}|\delta Y_{s-}|^{p-2}U_s(e)e\widetilde{\bar{\mu}^n}(\upd e,\upd s)-p\int_t^T\int_{\mathbb{R}}\delta Y_{s-}|\delta Y_{s-}|^{p-2}\delta U_s(e)e\widetilde{\mu^n}(\upd e,\upd s)\label{eq:mainproof1}\\
=&\ |\delta g(X_T)|^p+ \int_t^T p\delta Y_s|\delta Y_s|^{p-2}\delta f (\Theta_s)\upd s\\
&-p\int_t^T\delta Y_{s-}|\delta Y_{s-}|^{p-2}\delta Z_s\upd B_s-\frac{1}{2}\int_t^Tp(p-1)|\delta Y_s|^{p-2}\delta Z_s^2\upd s\\
&-\int_t^T\int_{\mathbb{R}}\left(|\delta Y_{s-}+U_s(e)e|^p-|\delta Y_{s-}|^p-p\delta Y_{s-}|\delta Y_{s-}|^{p-2}U_s(e)e\right)\bar{\nu}^n(\upd e)\upd s\\
&-\int_t^T\int_{\mathbb{R}}\left(|\delta Y_{s-}+\delta U_s(e)e|^p-|\delta Y_{s-}|^p-p\delta Y_{s-}|\delta Y_{s-}|^{p-2}\delta U_s(e)e\right)\nu^n(\upd e)\upd s\\
&-\int_t^T\int_{\mathbb{R}}\left(|\delta Y_{s-}+U_s(e)e|^p-|\delta Y_{s-}|^p\right)\widetilde{\bar{\mu}^n}(\upd e,\upd s)-\int_t^T\int_{\mathbb{R}}\left(|\delta Y_{s-}+\delta U_s(e)e|^p-|\delta Y_{s-}|^p\right)\widetilde{\mu^n}(\upd e,\upd s)
\end{align}

We use a Taylor expansion of $\eta(x+y)$ around $x$.
\begin{align}\label{eq:boundeta}
\eta(x+y)-\eta(x)-\frac{\partial\eta}{\partial x}(x)y&=p(p-1)\int_0^1(1-r)|x+ry|^{p-2}|y|^2\upd r\\
&\ge p(p-1)3^{1-p}|y|^2|x|^{p-2}.
\end{align}
The inequality follows by Lemma A.4 of \cite{Yao2010}, an earlier version of \cite{Yao2017}. For the readers' convenience we state it now: Let $p\in(0,\infty)$ and let $\mathbb{B}$ be a generic real Banach space with norm $|\cdot|_{\mathbb{B}}$. For any $x,y\in\mathbb{B}$,
\begin{equation}\label{eq:Yao2010}
\int_0^1(1-\alpha)|x+\alpha y|_{\mathbb{B}}^p\upd \alpha\ge 3^{-(1+p)}|x|_{\mathbb{B}}.
\end{equation} 

The above implies
\noindent
\begin{align}
&-\int_t^T\int_{\mathbb{R}}\left(|\delta Y_{s-}+U_s(e)e|^p-|\delta Y_{s-}|^p-p\delta Y_{s-}|\delta Y_{s-}|^{p-2}U_s(e)e\right)\bar{\nu}^n(\upd e)\upd s\\
&-\int_t^T\int_{\mathbb{R}}\left(|\delta Y_{s-}+\delta U_s(e)e|^p-|\delta Y_{s-}|^p-p\delta Y_{s-}|\delta Y_{s-}|^{p-2}\delta U_s(e)e\right)\nu^n(\upd e)\upd s\\
\le& -p(p-1)3^{1-p}\int_t^T\int_{\mathbb{R}}|\delta Y_{s-}|^{p-2}|U_s(e)|^2 e^2\bar{\nu}^n(\upd e)\upd s\\
&-p(p-1)3^{1-p}\int_t^T\int_{\mathbb{R}}|\delta Y_{s-}|^{p-2}|\delta U_s(e)|^2 e^2\nu^n(\upd e)\upd s\\
=& -p(p-1)3^{1-p}\int_t^T|\delta Y_{s}|^{p-2}||U_s||_{\mathbb{L}_{\bar{\nu}^n}^2}^2\upd s-p(p-1)3^{1-p}\int_t^T|\delta Y_{s}|^{p-2}||\delta U_s||_{\mathbb{L}_{\nu^n}^2}^2\upd s.
\end{align}

\noindent Denote $\kappa_p:=p(p-1)3^{1-p}$ and by $\kappa_p\le \frac{p(p-1)}{2}$ we get that \eqref{eq:mainproof1} becomes
\begin{align}
&|\delta Y_t|^p+\kappa_p\int_t^T|\delta Y_{s-}|^{p-2}\delta Z_s^2\upd s+ \kappa_p\int_t^T|\delta Y_{s}|^{p-2}||U_s||_{\mathbb{L}_{\bar{\nu}^n}^2}^2\upd s+\kappa_p\int_t^T|\delta Y_{s}|^{p-2}||\delta U_s||_{\mathbb{L}_{\nu^n}^2}^2\upd s\\
\le &\ |\delta g(X_T)|^p + \int_t^T p\delta Y_s|\delta Y_s|^{p-2}\delta f (\Theta_s)\upd s -p\int_t^T\delta Y_{s-}|\delta Y_{s-}|^{p-2}\delta Z_s\upd B_s\\
&-\int_t^T\int_{\mathbb{R}}\left(|\delta Y_{s-}+U_s(e)e|^p-|\delta Y_{s-}|^p\right)\widetilde{\bar{\mu}^n}(\upd e,\upd s)-\int_t^T\int_{\mathbb{R}}\left(|\delta Y_{s-}+\delta U_s(e)e|^p-|\delta Y_{s-}|^p\right)\widetilde{\mu^n}(\upd e,\upd s).
\end{align}

\noindent Now we apply the Lipschitz condition \eqref{eq:LipschitzY} of $f $ to obtain
\begin{align}
&|\delta Y_t|^p+\kappa_p\int_t^T|\delta Y_s|^{p-2}\delta Z_s^2\upd s+ \kappa_p\int_t^T|\delta Y_{s}|^{p-2}||U_s||_{\mathbb{L}_{\bar{\nu}^n}^2}^2\upd s+\kappa_p\int_t^T|\delta Y_{s}|^{p-2}||\delta U_s||_{\mathbb{L}_{\nu^n}^2}^2\upd s\\
\le &\ |\delta g(X_T)|^p + K p \int_t^T \delta Y_s|\delta Y_s|^{p-2} |\delta X_s| \upd s + K p \int_t^T \delta Y_s|\delta Y_s|^{p-1} \upd s + K p \int_t^T \delta Y_s|\delta Y_s|^{p-2} |\delta Z_s| \upd s\\
& + K p \int_t^T \delta Y_s|\delta Y_s|^{p-2} \int_{\mathbb{R}}\rho(e)|U_s(e)| e\bar{\nu}^n(\upd e) \upd s + K p \int_t^T \delta Y_s|\delta Y_s|^{p-2} \int_{\mathbb{R}}\rho(e)|\delta U_s(e)|e \nu^n(\upd e) \upd s\\
& - p\int_t^T\delta Y_{s-}|\delta Y_{s-}|^{p-2}\delta Z_s\upd B_s\\
& - \int_t^T\int_{\mathbb{R}}\left(|\delta Y_{s-}+U_s(e)e|^p-|\delta Y_{s-}|^p\right)\widetilde{\bar{\mu}^n}(\upd e,\upd s)-\int_t^T\int_{\mathbb{R}}\left(|\delta Y_{s-}+\delta U_s(e)e|^p-|\delta Y_{s-}|^p\right)\widetilde{\mu^n}(\upd e,\upd s).
\end{align}

\noindent Next we use the inequality $xy\le \alpha x^2+y^2/\alpha$ for $\alpha>0$, $x,y\ge0$, the bound of $\rho$ (recall Remark \ref{rem:boundonrho}), and that $\delta Y_s\le|\delta Y_s|$ to derive
\begin{align}
&|\delta Y_t|^p+\kappa_p\int_t^T|\delta Y_s|^{p-2}\delta Z_s^2\upd s+ \kappa_p\int_t^T|\delta Y_{s}|^{p-2}||U_s||_{\mathbb{L}_{\bar{\nu}^n}^2}^2\upd s+\kappa_p\int_t^T|\delta Y_{s}|^{p-2}||\delta U_s||_{\mathbb{L}_{\nu^n}^2}^2\upd s\\
\le &\ |\delta g(X_T)|^p + Kp(1+\alpha+\beta+\gamma+\varepsilon) \int_t^T |\delta Y_s|^{p} \upd s + \frac{K p}{\alpha} \int_t^T |\delta Y_s|^{p-2} |\delta X_s|^2 \upd s+ \frac{K p}{\beta} \int_t^T |\delta Y_s|^{p-2} |\delta Z_s|^2 \upd s\\
& + \frac{K^4 p}{\gamma} \int_t^T |\delta Y_s|^{p-2} || U_s||_{\mathbb{L}_{\bar{\nu}^n}^2}^2 \upd s
 + \frac{K^4 p}{\varepsilon} \int_t^T |\delta Y_s|^{p-2} ||\delta U_s||_{\mathbb{L}_{\nu^n}^2}^2 \upd s-p\int_t^T\delta Y_{s-}|\delta Y_{s-}|^{p-2}\delta Z_s\upd B_s\\
&-\int_t^T\int_{\mathbb{R}}\left(|\delta Y_{s-}+U_s(e)e|^p-|\delta Y_{s-}|^p\right)\widetilde{\bar{\mu}^n}(\upd e,\upd s)-\int_t^T\int_{\mathbb{R}}\left(|\delta Y_{s-}+\delta U_s(e)e|^p-|\delta Y_{s-}|^p\right)\widetilde{\mu^n}(\upd e,\upd s).
\end{align}

\noindent We make use of the Lipschitz condition \eqref{eq:LipschitzX} on $g$ and Young's inequality for $|\delta Y_s|^{p-2} |\delta X_s|^2$ 
to get
\begin{align}
&|\delta Y_t|^p+\kappa_p\int_t^T|\delta Y_s|^{p-2}\delta Z_s^2\upd s+ \kappa_p\int_t^T|\delta Y_{s}|^{p-2}||U_s||_{\mathbb{L}_{\bar{\nu}^n}^2}^2\upd s+\kappa_p\int_t^T|\delta Y_{s}|^{p-2}||\delta U_s||_{\mathbb{L}_{\nu^n}^2}^2\upd s\\
\le &\ K^p|\delta X_T|^p + Kp\left(1+\alpha+\beta+\gamma+\varepsilon+\frac{p-2}{\alpha p}\right) \int_t^T |\delta Y_s|^{p} \upd s\\
& + \frac{2K }{\alpha} \int_t^T |\delta X_s|^p \upd s
+ \frac{K p}{\beta} \int_t^T |\delta Y_s|^{p-2} |\delta Z_s|^2 \upd s + \frac{K^4p}{\gamma} \int_t^T |\delta Y_s|^{p-2}|| U_s||_{\mathbb{L}_{\bar{\nu}^n}^2}^2 \upd s \\
&+ \frac{K^4 p}{\varepsilon} \int_t^T |\delta Y_s|^{p-2} ||\delta U_s||_{\mathbb{L}_{\nu^n}^2}^2 \upd s-p\int_t^T\delta Y_{s-}|\delta Y_{s-}|^{p-2}\delta Z_s\upd B_s\\
&-\int_t^T\int_{\mathbb{R}}\left(|\delta Y_{s-}+U_s(e)e|^p-|\delta Y_{s-}|^p\right)\widetilde{\bar{\mu}^n}(\upd e,\upd s)\\
&-\int_t^T\int_{\mathbb{R}}\left(|\delta Y_{s-}+\delta U_s(e)e|^p-|\delta Y_{s-}|^p\right)\widetilde{\mu^n}(\upd e,\upd s),\label{eq:mainproof2}
\end{align}
where we choose the constants $\alpha,\beta,\gamma,\varepsilon>0$ arbitrarily such that $\frac{Kp}{\beta}<\kappa_p$, $\frac{K^4p}{\gamma}<\kappa_p$ and $\frac{K^4p}{\varepsilon}<\kappa_p$.

Note that the local martingales in \eqref{eq:mainproof2} are indeed true martingales which follows from the proof of the existence of the SDEs, see \cite{Kruse2016}. We then take expectations of \eqref{eq:mainproof2} to obtain
\begin{align}
& \mathbb{E}\left[|\delta Y_t|^p+\kappa_p\int_t^T|\delta Y_s|^{p-2}\delta Z_s^2\upd s+\kappa_p\int_t^T|\delta Y_{s}|^{p-2}||U_s||_{\mathbb{L}_{\bar{\nu}^n}^2}^2\upd s+\kappa_p\int_t^T|\delta Y_{s}|^{p-2}||\delta U_s||_{\mathbb{L}_{\nu^n}^2}^2\upd s\right]\\
\le& C_p\left(\sigma^p(n)+\sigma^2(n)^{p/2}+C_p \mathbb{E}\left[\int_t^T |\delta Y_s|^{p} \upd s\right]\right).\label{eq:mainproof3}
\end{align}


\noindent Then Gronwall's lemma implies
\begin{align}
\mathbb{E}\left[|\delta Y_t|^p\right]\le C_p \left(\sigma^p(n)+\sigma^2(n)^{p/2}\right).\label{eq:mainproof4}
\end{align}
We substitute \eqref{eq:mainproof4} into \eqref{eq:mainproof3} to get
\begin{align}
\mathbb{E}\left[\int_0^T|\delta Y_s|^{p-2}\delta Z_s^2\upd s+ \int_0^T|\delta Y_{s}|^{p-2}||U_s||_{\mathbb{L}_{\bar{\nu}^n}^2}^2\upd s+\int_0^T|\delta Y_{s}|^{p-2}||\delta U_s||_{\mathbb{L}_{\nu^n}^2}^2\upd s\right]\le C_p \left(\sigma^p(n)+\sigma^2(n)^{p/2}\right),
\end{align}
which implies that
\begin{align}
&\mathbb{E}\left[\int_0^T|\delta Y_s|^p\upd s+\int_0^T|\delta Y_s|^{p-2}\delta Z_s^2\upd s+ \int_0^T|\delta Y_{s}|^{p-2}||U_s||_{\mathbb{L}_{\bar{\nu}^n}^2}^2\upd s+\int_0^T|\delta Y_{s}|^{p-2}||\delta U_s||_{\mathbb{L}_{\nu^n}^2}^2\upd s\right]\\
\le&\ C_p \left(\sigma^p(n)+\sigma^2(n)^{p/2}\right),
\end{align}

Now we apply the Burkholder-Davis-Gundy inequality and Young's inequality to the martingales in \eqref{eq:mainproof1}. First,
\begin{align}
&\mathbb{E}\left[\sup_{0\le t\le T}\left|\int_t^T \delta Y_{s-} |\delta Y_{s-}|^{p-2}\delta Z_s\upd B_s\right|\right] \le C_p \mathbb{E}\left[\left(\int_0^T|\delta Y_{s}|^{2p-2}|\delta Z_s|^2\upd s\right)^{1/2}\right]\\
\le&\ \frac{1}{4p}\mathbb{E}\left[\sup_{0\le t\le T}|\delta Y_t|^p\right]+pC_p^2 \mathbb{E}\left[\int_0^T|\delta Y_s|^{p-2}|\delta Z_s|^2\upd s\right].\label{eq:mainbdg1}
\end{align}
Second, 
\begin{align}
&\mathbb{E}\left[\sup_{0\le t\le T}\left|\int_t^T\int_{\mathbb{R}}\delta Y_{s-}|\delta Y_{s-}|^{p-2}U_s(e)e\widetilde{\bar{\mu}^n}(\upd e,\upd s) \right|\right] \le  C_p \mathbb{E}\left[\left(\int_0^T\int_{\mathbb{R}}|\delta Y_{s-}|^{2p-2}U_s(e)^2e^2\bar{\mu}^n(\upd e,\upd s)\right)^{1/2}\right]\\
\le &\ \frac{1}{4p}\mathbb{E}\left[\sup_{0\le t\le T}|\delta Y_t|^p\right]+pC_p^2\mathbb{E}\left[\int_0^T|\delta Y_s|^{p-2}||U_s||_{\mathbb{L}_{\bar{\nu}^n}^2}^2\upd s\right].\label{eq:mainbdg2}
\end{align}
Third,
\begin{align}
&\mathbb{E}\left[\sup_{0\le t\le T}\left|\int_t^T\int_{\mathbb{R}}\delta Y_{s-}|\delta Y_{s-}|^{p-2}\delta U_s(e)e\widetilde{\mu^n}(\upd e,\upd s)\right|\right]\le C_p \mathbb{E}\left[\left(\int_0^T\int_{\mathbb{R}}|\delta Y_{s-}|^{2p-2}\delta U_s(e)^2e^2\mu^n(\upd e,\upd s)\right)^{1/2}\right]\\
\le &\ \frac{1}{4p}\mathbb{E}\left[\sup_{0\le t\le T}|\delta Y_t|^p\right]+pC_p^2\mathbb{E}\left[\int_0^T|\delta Y_s|^{p-2}||\delta U_s||_{\mathbb{L}_{\nu^n}^2}^2\upd s\right].\label{eq:mainbdg3}
\end{align}

We return to \eqref{eq:mainproof1} and use the bound \eqref{eq:boundeta} for $\eta$ to get
\begin{align}
&|\delta Y_t|^p+\kappa_p\int_t^T|\delta Y_s|^{p-2}\delta Z_s^2\upd s+ \kappa_p\int_t^T|\delta Y_{s}|^{p-2}||U_s||_{\mathbb{L}_{\bar{\nu}^n}^2}^2\upd s+\kappa_p\int_t^T|\delta Y_{s}|^{p-2}||\delta U_s||_{\mathbb{L}_{\nu^n}^2}^2\upd s\\
\le &\ |\delta g(X_T)|^p+ \int_t^T p\delta Y_s|\delta Y_s|^{p-2}\delta f (\Theta_s)\upd s-p\int_t^T\delta Y_{s-}|\delta Y_{s-}|^{p-2}\delta Z_s\upd B_s\\
&-p\int_t^T\int_{\mathbb{R}}\delta Y_{s-}|\delta Y_{s-}|^{p-2}U_s(e)e\widetilde{\bar{\mu}^n}(\upd e,\upd s)-p\int_t^T\int_{\mathbb{R}}\delta Y_{s-}|\delta Y_{s-}|^{p-2}\delta U_s(e)e\widetilde{\mu^n}(\upd e,\upd s).
\end{align}
When we now follow the previous lines in the proof to bound the $\delta f (\Theta_s)$ integral and use the bounds \eqref{eq:mainbdg1}, \eqref{eq:mainbdg2} and \eqref{eq:mainbdg3} by the Burkholder-Davis-Gundy inequality we finally derive
\begin{align}
\mathbb{E}\left[\sup_{0\le t\le T} |Y_t|^p\right] \le C_p \left(\sigma^p(n)+\sigma^2(n)^{p/2}\right).
\end{align}

\noindent \textbf{Step 2:}
In the second step we prove that
\begin{align}
&\mathbb{E}\left[\left(\int_0^T|\delta Z_s|^2\upd s\right)^{p/2} + \left(\int_0^T\int_{\mathbb{R}}|U_s(e)e|^2\bar{\nu}^n(\upd e)\upd s\right)^{p/2} + \left(\int_0^T\int_{\mathbb{R}}|\delta U_s(e)e|^2\nu^n(\upd e)\upd s\right)^{p/2}\right]\\
\le&\ C_p \left(\sigma^p(n)+\sigma^2(n)^{p/2}\right).\label{eq:secondstep}
\end{align}

\noindent Again we apply Itô's formula, this time to $|\delta Y_t|^2$:
\begin{align}
&|\delta Y_0|^2+\int_0^{T}|\delta Z_s|^2\upd s + \int_0^{T}\int_{\mathbb{R}}|U_s(e)e|^2\bar{\mu}^n(\upd e,\upd s) + \int_0^{T}\int_{\mathbb{R}}|\delta U_s(e)e|^2\mu^n(\upd e,\upd s) \\
=&\ |\delta Y_{T}|^2+2 \int_0^{T}\delta Y_s\delta f (\Theta_s)\upd s -2 \int_0^{T}\delta Y_s \delta Z_s\upd B_s\\
&-2 \int_0^{T}\int_{\mathbb{R}}\delta Y_{s-} U_s(e)e \widetilde{\bar{\mu}^n}(\upd e,\upd s) - 2 \int_0^{T}\int_{\mathbb{R}}\delta Y_{s-}\delta U_s(e)e\widetilde{\mu^n}(\upd e,\upd s).
\end{align}

\noindent Next we use the Lipschitz condition \eqref{eq:LipschitzY} 
\begin{align}
&\int_0^{T}|\delta Z_s|^2\upd s + \int_0^{T}\int_{\mathbb{R}}|U_s(e)e|^2\bar{\mu}^n(\upd e,\upd s) + \int_0^{T}\int_{\mathbb{R}}|\delta U_s(e)e|^2\mu^n(\upd e,\upd s) \\
\le&\ |\delta Y_{*}|^2+2K \int_0^{T}|\delta Y_s|^2\upd s +2K \int_0^{T}\delta Y_s |\delta X_s|\upd s +2K \int_0^{T}\delta Y_s|\delta Z_s|\upd s \\
&+2K \int_0^{T}\delta Y_s\int_{\mathbb{R}}\rho(e)| U_s(e)| e\bar{\nu}^n(\upd e)\upd s + 2K \int_0^{T}\delta Y_s\int_{\mathbb{R}}\rho(e)|\delta U_s(e)|e\nu^n(\upd e) \upd s \\ 
&-2 \int_0^{T}\delta Y_s \delta Z_s\upd B_s - 2 \int_0^{T}\int_{\mathbb{R}}\delta Y_{s-} U_s(e)e \widetilde{\bar{\mu}^n}(\upd e,\upd s) - 2 \int_0^{T}\int_{\mathbb{R}}\delta Y_{s-}\delta U_s(e)e\widetilde{\mu^n}(\upd e,\upd s),
\end{align}
where $\delta Y_*:=\sup_{0\le t\le T} |\delta Y_t|$.

\noindent We again use the inequality $xy\le \alpha x^2+y^2/\alpha$ for $\alpha>0$, $x,y\ge0$ to get the bound
\begin{align}
&\int_0^{T}|\delta Z_s|^2\upd s + \int_0^{T}\int_{\mathbb{R}}|U_s(e)e|^2\bar{\mu}^n(\upd e,\upd s) + \int_0^{T}\int_{\mathbb{R}}|\delta U_s(e)e|^2\mu^n(\upd e,\upd s)\\
\le&\ |\delta Y_{*}|^2+2K(1+ \alpha+\beta+\gamma+\varepsilon) \int_0^{T}|\delta Y_s|^2\upd s +\frac{2K}{\alpha} \int_0^{T} |\delta X_s|^2\upd s +\frac{2K}{\beta} \int_0^{T}|\delta Z_s|^2\upd s \\
&+\frac{2K^4}{\gamma} \int_0^{T}|| U_s||_{\mathbb{L}_{\bar{\nu}^n}^2}^2\upd s +\frac{2K^4}{\varepsilon} \int_0^{T}||\delta U_s||_{\mathbb{L}_{\nu^n}^2}^2 \upd s
-2 \int_0^{T}\delta Y_s \delta Z_s\upd B_s\\
&-2 \int_0^{T}\int_{\mathbb{R}}\delta Y_{s-} U_s(e)e \widetilde{\bar{\mu}^n}(\upd e,\upd s) - 2 \int_0^{T}\int_{\mathbb{R}}\delta Y_{s-}\delta U_s(e)e\widetilde{\mu^n}(\upd e,\upd s).\label{eq:mainproof5}
\end{align}

\noindent Next we take powers of \eqref{eq:mainproof5} (and use Jensen's inequality for the first two integrals on the RHS)
\begin{align}
&\left(\int_0^{T}|\delta Z_s|^2\upd s\right)^{p/2} + \left(\int_0^{T}\int_{\mathbb{R}}|U_s(e)e|^2\bar{\mu}^n(\upd e,\upd s)\right)^{p/2} + \left(\int_0^{T}\int_{\mathbb{R}}|\delta U_s(e)e|^2\mu^n(\upd e,\upd s)\right)^{p/2}\\
\le&\ C_p|\delta Y_{*}|^p+C_p\left(2K(1+ \alpha+\beta+\gamma+\varepsilon)\right)^{p/2} \int_0^{T}|\delta Y_s|^p\upd s \\
&+C_p\left(\frac{2K}{\alpha}\right)^{p/2} \int_0^{T} |\delta X_s|^p\upd s  +C_p\left(\frac{2K}{\beta}\right)^{p/2} \left(\int_0^{T}|\delta Z_s|^2\upd s\right)^{p/2}\\
&+C_p\left(\frac{2K^4}{\gamma}\right)^{p/2} \left(\int_0^{T}|| U_s||_{\mathbb{L}_{\bar{\nu}^n}^2}^2\upd s\right)^{p/2} +C_p\left(\frac{2K^4}{\varepsilon}\right)^{p/2} \left(\int_0^{T}||\delta U_s||_{\mathbb{L}_{\nu^n}^2}^2 \upd s\right)^{p/2} \\
&+C_p\left( \left|\int_0^{T}\delta Y_s \delta Z_s\upd B_s\right|^{p/2}
+\left| \int_0^{T}\int_{\mathbb{R}}\delta Y_{s-} U_s(e)e \widetilde{\bar{\mu}^n}(\upd e,\upd s)\right|^{p/2} + \left|\int_0^{T}\int_{\mathbb{R}}\delta Y_{s-}\delta U_s(e)e\widetilde{\mu^n}(\upd e,\upd s)\right|^{p/2}\right).\label{eq:mainproof6}
\end{align}

\noindent Because $p/2\ge1$, we can apply the Burkholder-Davis-Gundy inequality and Young's inequality to get
\begin{align}
&C_p\mathbb{E}\left[\left|\int_0^{T}\delta Y_s \delta Z_s\upd B_s\right|^{p/2}\right]\le c_p \mathbb{E}\left[\left(\int_0^{T}|\delta Y_s|^2|\delta Z_s|^2\upd s\right)^{p/4}\right]\le \frac{c_p^2}{4}\mathbb{E}\left[|\delta Y_*|^p\right]+\frac{1}{2}\mathbb{E}\left[\left(\int_0^{T}|\delta Z_s|^2\upd s\right)^{p/2}\right],\\
&C_p \mathbb{E}\left[\left|\int_0^{T}\int_{\mathbb{R}}\delta Y_{s-} U_s(e)e\widetilde{\bar{\mu}^n}(\upd e,\upd s)\right|^{p/2}\right] \le c_p \mathbb{E}\left[\left(\int_0^{T}\int_{\mathbb{R}}|\delta Y_{s-}|^2|U_s(e)|^2e^2\bar{\mu}^n(\upd e,\upd s)\right)^{p/4}\right]\\
& \ \ \ \le \frac{c_p^2}{4}\mathbb{E}\left[|\delta Y_*|^p\right]+\frac{1}{2}\mathbb{E}\left[\left(\int_0^{T}||U_s||_{\mathbb{L}_{\bar{\nu}^n}^2}^2\bar{\mu}^n(\upd e,\upd s)\right)^{p/2}\right],\\
&C_p \mathbb{E}\left[\left|\int_0^{T}\int_{\mathbb{R}}\delta Y_{s-}\delta U_s(e)e\widetilde{\mu^n}(\upd e,\upd s)\right|^{p/2}\right] \le c_p \mathbb{E}\left[\left(\int_0^{T}\int_{\mathbb{R}}|\delta Y_{s-}|^2|\delta U_s(e)|^2e^2\mu^n(\upd e,\upd s)\right)^{p/4}\right]\\
& \ \ \ \le \frac{c_p^2}{4}\mathbb{E}\left[|\delta Y_*|^p\right]+\frac{1}{2}\mathbb{E}\left[\left(\int_0^{T}||\delta U_s||_{\mathbb{L}_{\nu^n}^2}^2\mu^n(\upd e,\upd s)\right)^{p/2}\right],
\end{align}
for some constant $c_p$. Using this for the expectation of \eqref{eq:mainproof6}, we see
\begin{align}
&\frac{1}{2}\mathbb{E}\left[\left(\int_0^{T}|\delta Z_s|^2\upd s\right)^{p/2}\right] +\frac{1}{2}\mathbb{E}\left[\left(\int_0^{T}\int_{\mathbb{R}}|U_s(e)e|^2\bar{\mu}^n(\upd e,\upd s)\right)^{p/2}\right]\\
&\ + \frac{1}{2}\mathbb{E}\left[ \left(\int_0^{T}\int_{\mathbb{R}}|\delta U_s(e)e|^2\mu^n(\upd e,\upd s)\right)^{p/2}\right]\\
\le &\ C_{p,K,T,\alpha,\beta,\gamma,\varepsilon}\mathbb{E}\left[|\delta Y_*|^p\right]+C_{p,K,T,\alpha}\mathbb{E}\left[\sup_{0\le t\le T}|\delta X_t|^p\right]+C_p\left(\frac{2K}{\beta}\right)^{p/2} \mathbb{E}\left[\left(\int_0^{T}|\delta Z_s|^2\upd s\right)^{p/2}\right]\\
&+C_p\left(\frac{2K^4}{\gamma}\right)^{p/2}\mathbb{E}\left[\left(\int_0^{T}|| U_s||_{\mathbb{L}_{\bar{\nu}^n}^2}^2\upd s\right)^{p/2}\right]  +C_p\left(\frac{2K^4}{\varepsilon}\right)^{p/2}\mathbb{E}\left[\left(\int_0^{T}||\delta U_s||_{\mathbb{L}_{\nu^n}^2}^2 \upd s\right)^{p/2}\right].
\end{align}

\noindent As in \cite{Kruse2016} \cite[see also][]{Dzhaparidze1990} we use the bounds
\begin{align}
\mathbb{E}\left[\left(\int_0^{T}|| U_s||_{\mathbb{L}_{\bar{\nu}^n}^2}^2\upd s\right)^{p/2}\right]&\le d_p\mathbb{E}\left[\left(\int_0^{T}\int_{\mathbb{R}}|U_s(e)e|^2\bar{\mu}^n(\upd e,\upd s)\right)^{p/2}\right],\\
\mathbb{E}\left[\left(\int_0^{T}||\delta U_s||_{\mathbb{L}_{\nu^n}^2}^2 \upd s\right)^{p/2}\right]&\le d_p\mathbb{E}\left[ \left(\int_0^{T}\int_{\mathbb{R}}|\delta U_s(e)e|^2\mu^n(\upd e,\upd s)\right)^{p/2}\right],
\end{align}
for some constant $d_p>0$.

\noindent All in all we can choose the constants $\alpha$, $\beta$, $\gamma$ and $\varepsilon$ (only depending on $p$) such that
\begin{align}
&\mathbb{E}\left[\left(\int_0^{T}|\delta Z_s|^2\upd s\right)^{p/2}\right] + \mathbb{E}\left[\left(\int_0^{T}|| U_s||_{\mathbb{L}_{\bar{\nu}^n}^2}^2\upd s\right)^{p/2}\right] + \mathbb{E}\left[\left(\int_0^{T}||\delta U_s||_{\mathbb{L}_{\nu^n}^2}^2 \upd s\right)^{p/2}\right]\\
\le&\ C_{p}\mathbb{E}\left[|\delta Y_*|^p\right]+C_{p}\mathbb{E}\left[\sup_{0\le t\le T}|\delta X_t|^p\right]\\
\le&\ C_p \left(\sigma^p(n)+\sigma^2(n)^{p/2}\right),
\end{align}
by Step 1 and Proposition \ref{prop:approxX}.
\end{proof}

\section{Error of the discretization of the FBSDE}\label{sec:discrete}
In this section we discretize the approximated FBSDE $(X^n,Y^n,Z^n,U^n)$ and derive error bounds. We use a forward-backward Euler scheme for simulation.
First we define the regular grid $\tilde\pi:=\left\{\tilde t_k:=\frac{kT}{N},\ k=0,\ldots,N\right\}$ on $[0,T]$. 
We do not discuss the discretization of the original FBSDE because in practice they cannot be simulated and the proofs of this section rely on $\nu^n(\mathbb{R})<\infty$.
Starting with the forward Euler scheme for $X^n$, we define
\begin{align}\label{eq:EulerSchemeX}
\begin{cases}
X_0^{n,\pi}&:= X_0\\
X_{\tilde t_{k+1}}^{n,\pi}&:=X_{\tilde t_{k}}^{n,\pi}+\frac{T}{N}b(\tilde t_k,X_{t_{k}}^{n,\pi})+a(\tilde t_k,X_{t_{k}}^{n,\pi})\Delta B_{{k+1}}+\int_{\mathbb{R}}h(\tilde t_k,X_{\tilde t_{k}}^{n,\pi})e\widetilde{\mu^n}(\upd e, (\tilde t_k,\tilde t_{k+1}]),
\end{cases}
\end{align}
where $\Delta B_{{k+1}}:=B_{\tilde t_{k+1}}-B_{\tilde t_k}$ are normal random variables.
We now aim to to derive the discretization error of the forward SDE. Although the techniques are pretty standard, we have to reconsider the result of \cite{Aazizi2013} because the discretization error depends on the truncation parameter $n$. We here follow \cite{Mrad2023} who derived a new approximation-discretization error for a jump-adapted scheme. In order to do so, we define the jump-adapted discretization scheme as
\begin{equation}
\pi:=\{ t_k,k=1,\ldots, N+ J^n\}:=\tilde\pi \cup \{T_i,i:G_i\le nT\},
\end{equation}
which is the superposition of the discretization times and the times when large jumps occur, and $J^n:=|\{T_i,i:G_i\le nT\}|$ denotes the number of large jumps. 

Let us define the function $\tau$ for $t\in[0,T]$ by
\begin{equation}\label{eq:psi}
\tau_t:=\max\{t_k,\ k=1,\ldots,N+J^n|t_k\le t\},
\end{equation}
and let
\begin{equation}\label{eq:EulerSchemeXcont}
X_{t}^{n,\pi}:=X_{\tau_t}^{n,\pi}+b(\tau_t,X_{\tau_t}^{n,\pi})(t-\tau_t)+a(\tau_t,X_{\tau_t}^{n,\pi})(B_t-B_{\tau_t})+\int_{\mathbb{R}}h(\tau_t,X_{\tau_t}^{n,\pi})e\widetilde{\mu^n}(\upd e, (\tau_t,t])
\end{equation}
be the jump-adapted Euler-scheme for $t\in(t_k,t_{k+1}]$.
This can be written as
\begin{equation}\label{eq:EulerSchemeXint}
X_{t}^{n,\pi}=X_{0}+\int_0^tb(\tau_s,X_{\tau_s}^{n,\pi})\upd s+\int_0^ta(\tau_s,X_{\tau_s}^{n,\pi})\upd B_s+\int_0^t\int_{\mathbb{R}}h(\tau_s,X_{\tau_s}^{n,\pi})e\widetilde{\mu^n}(\upd e, \upd s)
\end{equation}
for $t\in[0,T]$.
We now state estimates for the discretization error of the forward SDE. The following bound is derived in Lemma 1 in \cite{Mrad2023}. Let us define
\begin{equation}
\label{eq:skriptm}
\mathfrak{m}^p(n):=\int_{\mathbb{R}}|e|^p\nu^n(\upd e).
\end{equation}
Under Assumption \ref{ass:Lipschitz}.(i) on $(\Omega,\mathcal{F},(\mathcal{F}_t),\mathbb{P})$ there exists a constant $C_p$ such that
\begin{equation}
\mathbb{E}\left[\sup_{\tau_u\le u\le t}|X_u^{n,\pi}-X_{\tau_u}^{n,\pi}|^p\right]\le C_p \left(N^{-p/2}+N^{-p}\mathfrak{m}^1(n)^{p}\right),
\end{equation}
for any $t\in[0,T]$. Then the following lemma follows analogously to Theorem 6 in \cite{Mrad2023}.
\begin{lem}\label{lem:EulerErrX}
Let $p\ge2$. Under Assumption \ref{ass:Lipschitz}.(i) on $(\Omega,\mathcal{F},(\mathcal{F}_t),\mathbb{P})$ there exists a constant $C_p$ such that the Euler scheme \eqref{eq:EulerSchemeX} of the forward SDE has the discretization error
\begin{equation}\label{eq:forwardEulerError}
\mathbb{E}\left[\sup_{t\in[0,T]}|X_t^n-X_t^{n,\pi}|^p\right]\le C_p \left(N^{-p/2}+N^{-p}\mathfrak{m}^1(n)^{p}\right),
\end{equation}
for all $p\ge2$.
\end{lem}

Next we introduce the backward implicit scheme to approximate $(Y^n,Z^n,\Gamma^n)$. We follow \cite{Bouchard2008} and \cite{Elie2006} and define
\begin{align}\label{eq:EulerSchemeY}
\begin{cases}
\bar{Z}_t^{n,\pi}&:=(t_{k+1}-t_k)^{-1}\mathbb{E}\left[\bar{Y}_{t_{k+1}}^{n,\pi}(B_{t_{k+1}}-B_{t_k})|\mathcal{F}_{t_k}\right]\\
\bar{\Gamma}_t^{n,\pi}&:=(t_{k+1}-t_k)^{-1}\mathbb{E}\left[\bar{Y}_{t_{k+1}}^{n,\pi}\int_{\mathbb{R}}\rho(e)e\widetilde{\mu^n}(\upd e,(t_k,t_{k+1}])|\mathcal{F}_{t_k}\right]\\
\bar{Y}_t^{n,\pi}&:=\mathbb{E}\left[\bar{Y}_{t_{k+1}}^{n,\pi}|\mathcal{F}_{t_k}\right]+(t_{k+1}-t_k)f\left(t_k,X_{t_k}^{n,\pi},\bar{Y}_{t_k}^{n,\pi},\bar{Z}_{t_k}^{n,\pi},\bar{\Gamma}_{t_k}^{n,\pi}\right),
\end{cases}
\end{align}
on each interval $[t_k,t_{k+1})$, where $Y_{T}^{n,\pi}:=g(X_{T}^{n,\pi})$. If $f$ depends on $Y^n$, the last step of \eqref{eq:EulerSchemeY} requires a fixed point procedure. However, since $f$ is Lipschitz continuous in the $y$ variable and because $f$ is multiplied by a value close to $1/(N+J^n)$ the approximation error can be neglected for large values of $N$ and $n$.

Given the backward scheme \eqref{eq:EulerSchemeY}, we will analyze the discretization error
\begin{equation}
Err_{\pi,p}(Y^n,Z^n,U^n):=\left(\sup_{0\le t\le T} \mathbb{E}\left[|Y_t^n-\bar{Y}_t^{n,\pi}|^p\right]+||Z^n-\bar{Z}^{n,\pi}||_{\mathbb{H}^p}^p+||\Gamma^n-\bar{\Gamma}^{n,\pi}||_{\mathbb{H}^p}^p\right)^{1/p}
\end{equation}
and we will show that it converges to zero with order $N^{-1/2}$.

In the following we discuss some related processes which will be needed throughout the proofs. By the representation theorem, see, e.g., Lemma 2.3 \cite{Tang1994}, there exist two processes $Z^{n,\pi}\in\mathbb{H}^p$ and $U^{n,\pi}\in\mathbb{L}_{\mu^n}^p$ such that
\begin{equation}
\bar{Y}_{t_{k+1}}^{n,\pi}-\mathbb{E}\left[\bar{Y}_{t_{k+1}}^{n,\pi}|\mathcal{F}_{t_k}\right]=\int_{t_k}^{t_{k+1}}Z_s^{n,\pi}\upd B_s+\int_{t_k}^{t_{k+1}}\int_{\mathbb{R}}U_s^{n,\pi}(e)e\bar{\mu}^n(\upd e,\upd s).
\end{equation}
Observe that $\bar{Z}_t^{n,\pi}$ and $\bar{\Gamma}_t^{n,\pi}$ in \eqref{eq:EulerSchemeY} satisfy
\begin{align}
\bar{Z}_{t_k}^{n,\pi}&=(t_{k+1}-t_k)^{-1}\mathbb{E}\left[\left.\int_{t_k}^{t_{k+1}}Z_s^{n,\pi}\upd s\right|\mathcal{F}_{t_k}\right],\label{eq:equalityZnpi}\\
\bar{\Gamma}_{t_k}^{n,\pi}&=(t_{k+1}-t_k)^{-1}\mathbb{E}\left[\left.\int_{t_k}^{t_{k+1}}\Gamma_s^{n,\pi}\upd s\right|\mathcal{F}_{t_k}\right],
\end{align}
and thus coincide with the best $\mathbb{H}_{[t_k,t_{k+1}]}^2$-approximations of the processes $(Z_t^{n,\pi})$ and $(\Gamma_t^{n,\pi}):=$ \\$\left(\int_{\mathbb{R}}\rho(e)U_t^{n,\pi}(e)e\nu^n(\upd e)\right)$ on $[t_k,t_{k+1})$ by $\mathcal{F}_{t_k}$-measurable random variables (viewed as constant processes on $[t_k,t_{k+1})$), i.e.,
\begin{align}
\mathbb{E}\left[\int_{t_k}^{t_{k+1}}|Z_s^{n,\pi}-\bar{Z}_{t_k}^{n,\pi}|^2\upd s\right]=\inf_{Z_k\in L^2(\Omega,\mathcal{F}_{t_k})}\mathbb{E}\left[\int_{t_k}^{t_{k+1}}|Z_s^{n,\pi}-Z_k|^2\upd s\right],\\
\mathbb{E}\left[\int_{t_k}^{t_{k+1}}|\Gamma_s^{n,\pi}-\bar{\Gamma}_{t_k}^{n,\pi}|^2\upd s\right]=\inf_{\Gamma_k\in L^2(\Omega,\mathcal{F}_{t_k})}\mathbb{E}\left[\int_{t_k}^{t_{k+1}}|\Gamma_s^{n,\pi}-\Gamma_k|^2\upd s\right].
\end{align}
Thus, it holds that
\begin{equation}
\bar{Y}_{t_{k}}^{n,\pi}=\bar{Y}_{t_{k+1}}^{n,\pi}+(t_{k+1}-t_k)f(t_k,X_{t_k}^{n,\pi},\bar{Y}_{t_{k}}^{n,\pi},\bar{Z}_{t_k}^{n,\pi},\bar{\Gamma}_{t_k}^{n,\pi})-\int_{t_k}^{t_{k+1}}Z_s^{n,\pi}\upd B_s-\int_{t_k}^{t_{k+1}}\int_{\mathbb{R}}U_s^{n,\pi}(e)e\bar{\mu}^n(\upd e,\upd s).
\end{equation}
We define the process $Y^{n,\pi}$
\begin{equation}
Y_{t}^{n,\pi}:=\bar{Y}_{t_{k}}^{n,\pi}-(t-t_k)f(t_k,X_{t_k}^{n,\pi},\bar{Y}_{t_{k}}^{n,\pi},\bar{Z}_{t_k}^{n,\pi},\bar{\Gamma}_{t_k}^{n,\pi})+\int_{t_k}^{t}Z_s^{n,\pi}\upd B_s+\int_{t_k}^{t}\int_{\mathbb{R}}U_s^{n,\pi}(e)e\bar{\mu}^n(\upd e,\upd s)
\end{equation}
on $[t_k,t_{k+1})$ and obtain that
\begin{equation}
(t_{k+1}-t_k)^{-1}\mathbb{E}\left[\int_{t_k}^{t_{k+1}}Y_s^{n,\pi}\upd s|\mathcal{F}_{t_k}\right]=\mathbb{E}\left[\bar{Y}_{t_{k+1}}^{n,\pi}|\mathcal{F}_{t_k}\right]+(t_{k+1}-t_k)f(t_k,X_{t_k}^{n,\pi},\bar{Y}_{t_{k}}^{n,\pi},\bar{Z}_{t_k}^{n,\pi},\bar{\Gamma}_{t_k}^{n,\pi})=Y_{t_{k}}^{n,\pi}=\bar{Y}_{t_{k}}^{n,\pi}.\label{eq:YtkbarYtk}
\end{equation}
Thus $\bar{Y}_{t_{k}}^{n,\pi}$ is the best approximation of $Y^{n,\pi}$ on $[t_k,t_{k+1})$ by $\mathcal{F}_{t_k}$-measurable random variables (viewed as constant processes on $[t_k,t_{k+1})$), which explains the notation $\bar{Y}^{n,\pi}$, consistent with the definition of $\bar{Z}^{n,\pi}$ and $\bar{\Gamma}^{n,\pi}$.

Furthermore, we need to define the processes $(\bar{Z}^n,\bar{\Gamma}^n)$ on each interval $[t_k,t_{k+1})$ by
\begin{align}
\bar{Z}_t^n&:=(t_{k+1}-t_k)\mathbb{E}\left[\left.\int_{t_k}^{t_{k+1}}Z_s^n\upd s\right|\mathcal{F}_{t_k}\right],\label{eq:defZnbar}\\
\bar{\Gamma}_t^n&:=(t_{k+1}-t_k)\mathbb{E}\left[\left.\int_{t_k}^{t_{k+1}}\Gamma_s^n\upd s\right|\mathcal{F}_{t_k}\right].
\end{align}

\begin{rem}
$\bar{Z}_{t_k}^n$ and $\bar{\Gamma}_{t_k}^n$ are the counterparts of $\bar{Z}_{t_k}^{n,\pi}$ and $\bar{\Gamma}_{t_k}^{n,\pi}$ for the original backward SDE. They can be interpreted as the best $\mathbb{H}_{[t_k,t_{k+1}]}^2$-approximations of $(Z_t^n)_{t_k\le t<t_{k+1}}$ and $(\Gamma_t^n)_{t_k\le t<t_{k+1}}$ by an $\mathcal{F}_{t_k}$-measurable random variables (viewed as constant processes on $[t_k,t_{k+1})$), i.e.,
\begin{align}
\mathbb{E}\left[\int_{t_k}^{t_{k+1}}|Z_s^{n}-\bar{Z}_{t_k}^{n}|^2\upd s\right]=\inf_{Z_k\in L^2(\Omega,\mathcal{F}_{t_k})}\mathbb{E}\left[\int_{t_k}^{t_{k+1}}|Z_s^{n}-Z_k|^2\upd s\right],\\
\mathbb{E}\left[\int_{t_k}^{t_{k+1}}|\Gamma_s^{n}-\bar{\Gamma}_{t_k}^{n}|^2\upd s\right]=\inf_{\Gamma_k\in L^2(\Omega,\mathcal{F}_{t_k})}\mathbb{E}\left[\int_{t_k}^{t_{k+1}}|\Gamma_s^{n}-\Gamma_k|^2\upd s\right].
\end{align}
\end{rem}

We now state our second main theorem, which gives a bound for the discretization error.
\begin{thm}\label{thm:EulerErr}
Under Assumptions \ref{ass:Lipschitz} and \ref{ass:invofh}, the discretization error for the backward SDE is bounded by
\begin{equation}
\label{eq:EulerErr}
Err_{\pi,p}(Y^n,Z^n,U^n)\le C_p \left(N^{-1/2}+N^{-1}\mathfrak{m}^1(n)\right),
\end{equation}
where $C_p$ only depends on constants and not on $n$ or $\pi$.
\end{thm}

\begin{proof}
The proof is an $L^p$ extension of the proofs of \cite{Bouchard2008}, \cite{Elie2006} and \cite{Bouchard2004}. For the sake of brevity we set $\delta^nY_t:=Y_t^n-Y_t^{n,\pi}$, $\delta^nZ_t:=Z_t^n-Z_t^{n,\pi}$, $\delta^nU_t(e):=U_t^n(e)-U_t^{n,\pi}(e)$, $\delta^n\Gamma_t:=\Gamma_t^{n}-\Gamma_t^{n,\pi}$ and $\delta^nf(\Theta_t):=f(t,X_t^n,Y_t^n,Z_t^n,\Gamma_t^n)-f(t_k,X_{t_k}^{n,\pi},\bar{Y}_{t_k}^{n,\pi},\bar{Z}_{t_k}^{n,\pi},\bar{\Gamma}_{t_k}^{n,\pi})$. Note that $\bar{Y}_{t_k}^{n,\pi}=Y_{t_k}^{n,\pi}$ by \eqref{eq:YtkbarYtk} which we will use repeatedly.

The proof is divided in four steps. Before turning to the first step, we discuss some bounds which we will need throughout. Let $s\in[t_k,t_{k+1})$. Then
\begin{equation}\label{eq:bdXdiscr}
\mathbb{E}\left[|X_s^n-X_{t_k}^{n,\pi}|\right]\le C_p \left(N^{-p/2}+N^{-p}\mathfrak{m}^1(n)^p\right),
\end{equation}
by \eqref{eq:forwardEulerError}. Moreover,
\begin{equation}
\mathbb{E}\left[|Y_s^n-\bar{Y}_{t_k}^{n,\pi}|^p\right]\le p\left(\mathbb{E}\left[|Y_s^n-Y_{t_k}^n|^p\right]+\mathbb{E}\left[|\delta^nY_{t_k}|^p\right]\right)
\end{equation}
and
\begin{align}
\mathbb{E}\left[|Z_s^n-\bar{Z}_{t_k}^{n,\pi}|^p\right]&\le p\left(\mathbb{E}\left[|Z_s^{n}-\bar{Z}_{t_k}^{n}|^p\right]+\mathbb{E}\left[|\bar{Z}_{t_k}^n-\bar{Z}_{t_k}^{n,\pi}|^p\right]\right)\\
&= p\left(\mathbb{E}\left[|Z_s^{n}-\bar{Z}_{t_k}^{n}|^p\right]+\mathbb{E}\left[\left|(t_{k+1}-t_k)^{-1}\mathbb{E}\left[\left.\int_{t_k}^{t_{k+1}}\delta^nZ_s\upd s\right|\mathcal{F}_{t_k}\right]\right|^p\right]\right)\\
&\le C_p\left(\mathbb{E}\left[|Z_s^{n}-\bar{Z}_{t_k}^{n}|^p\right]+ \mathbb{E}\left[\mathbb{E}\left[\left.(t_{k+1}-t_k)^{-1}\int_{t_k}^{t_{k+1}}|\delta^nZ_s|^2\upd s\right|\mathcal{F}_{t_k}\right]^{p/2}\right]\right)\\
&\le C_p \left(\mathbb{E}\left[|Z_s^{n}-\bar{Z}_{t_k}^{n}|^p\right]+ (t_{k+1}-t_k)^{-p/2}\mathbb{E}\left[\left(\int_{t_k}^{t_{k+1}}|\delta^nZ_s|^2\upd s\right)^{p/2}\right]\right)\\
&= C_p \left(\mathbb{E}\left[|Z_s^{n}-\bar{Z}_{t_k}^{n}|^p\right]+ (t_{k+1}-t_k)^{-p/2}||\delta^nZ||_{\mathbb{H}_{[t_k,t_{k+1}]}^p}^p\right)\label{eq:ZbarHp}.
\end{align}
The second equality follows by \eqref{eq:equalityZnpi} and \eqref{eq:defZnbar} and the third and fourth inequalities by Jensen's inequality. Analogously, using the bound on $\rho$, we can prove
\begin{equation}
\mathbb{E}\left[|\Gamma_s^n-\bar{\Gamma}_{t_k}^{n,\pi}|^p\right]\le C_p \left(\mathbb{E}\left[|\Gamma_s^{n}-\bar{\Gamma}_{t_k}^{n}|^p\right]+ (t_{k+1}-t_k)^{-p/2}||\delta^nU||_{\mathbb{L}_{\mu^n,[t_k,t_{k+1}]}^p}^p\right).\label{eq:GammabarHp}
\end{equation}

\noindent \textbf{Step 1:} 
We apply Itô's formula to $|\delta^n Y_t|^p$ for $t\in[t_k,t_{k+1})$,
\begin{align}
\mathbb{E}\left[|\delta^n Y_t|^p\right]=&\ \mathbb{E}\left[|\delta^n Y_{t_{k+1}}|^p\right]+p\mathbb{E}\left[\int_t^{t_{k+1}}\delta^nY_s\ |\delta^n Y_s|^{p-2}\delta^nf(\Theta_s)\upd s\right]\\
& - \frac{p(p-1)}{2}\mathbb{E}\left[\int_t^{t_{k+1}}|\delta^nY_s|^{p-2}|\delta^nZ_s|^2\upd s\right]\\
&-\mathbb{E}\left[\int_t^{t_{k+1}}\int_{\mathbb{R}}\left(|\delta^nY_{s-}+\delta^nU_{s}(e)e|^p-|\delta^nY_{s-}|^p-p\delta^nY_{s-}\ |\delta^n Y_{s-}|^{p-2}\delta^nU_{s}(e)e\right)\nu^n(\upd e)\upd s\right].
\end{align}
As in the proof of Theorem \ref{thm:approxY} we use
\begin{align}
&\ - \mathbb{E}\left[\int_t^{t_{k+1}}\int_{\mathbb{R}}\left(|\delta^nY_{s-}+\delta^nU_{s}(e)e|^p-|\delta^nY_{s-}|^p-p\delta^nY_{s-}\ |\delta^n Y_{s-}|^{p-2}\delta^nU_{s}(e)e\right)\nu^n(\upd e)\upd s\right]\\
\le &\ - \kappa_p\mathbb{E}\left[\int_t^{t_{k+1}}\int_{\mathbb{R}}|\delta^nY_s|^{p-2}|\delta^nU_s(e)e|^2\nu^n(\upd e)\upd s\right],
\end{align}
with $\kappa_p=p(p-1)3^{1-p}$, to derive
\begin{align}
&\ \mathbb{E}\left[|\delta^n Y_t|^p\right]+\kappa_p \mathbb{E}\left[\int_t^{t_{k+1}}|\delta^nY_s|^{p-2}|\delta^nZ_s|^2\upd s\right]+ \kappa_p\mathbb{E}\left[\int_t^{t_{k+1}}\int_{\mathbb{R}}|\delta^nY_s|^{p-2}|\delta^nU_s(e)e|^2\nu^n(\upd e)\upd s\right]\\
\le &\ \mathbb{E}\left[|\delta^n Y_{t_{k+1}}|^p\right]+ p\mathbb{E}\left[\int_t^{t_{k+1}}\delta^nY_s\ |\delta^n Y_s|^{p-2}\delta^nf(\Theta_s)\upd s\right].
\end{align}
We use the Lipschitz condition \eqref{eq:LipschitzY} and that $t_{k+1}-t_k\le T/N$ to get
\begin{align}
&\ \mathbb{E}\left[|\delta^n Y_t|^p\right]+\kappa_p \mathbb{E}\left[\int_t^{t_{k+1}}|\delta^nY_s|^{p-2}|\delta^nZ_s|^2\upd s\right]+ \kappa_p\mathbb{E}\left[\int_t^{t_{k+1}}\int_{\mathbb{R}}|\delta^nY_s|^{p-2}|\delta^nU_s(e)e|^2\nu^n(\upd e)\upd s\right]\\
\le & \ \mathbb{E}\left[|\delta^n Y_{t_{k+1}}|^p\right]\\
&+ \mathbb{E}\left[\int_t^{t_{k+1}}\delta^nY_s\ |\delta^n Y_s|^{p-2}\left((T/N)^{1/2}+|X_s^{n}-X_{t_k}^{n,\pi}|+|Y_s^{n}-\bar{Y}_{t_k}^{n,\pi}|+|Z_s^{n}-\bar{Z}_{t_k}^{n,\pi}|+|\Gamma_s^{n}-\bar{\Gamma}_{t_k}^{n,\pi}|\right)\upd s\right].
\end{align}
We rewrite this inequality to have
\begin{align}
&\ \mathbb{E}\left[|\delta^n Y_t|^p\right]+\kappa_p \mathbb{E}\left[\int_t^{t_{k+1}}|\delta^nY_s|^{p-2}|\delta^nZ_s|^2\upd s\right]+ \kappa_p\mathbb{E}\left[\int_t^{t_{k+1}}\int_{\mathbb{R}}|\delta^nY_s|^{p-2}|\delta^nU_s(e)e|^2\nu^n(\upd e)\upd s\right]\\
\le & \ \mathbb{E}\left[|\delta^n Y_{t_{k+1}}|^p\right]\\
&+ \mathbb{E}\left[\int_t^{t_{k+1}}\delta^nY_s\ |\delta^n Y_s|^{p-2}\left((T/N)^{1/2}+|X_s^{n}-X_{t_k}^{n,\pi}|+|\delta^n Y_{t_k}|+|Y_s^{n}-Y_{t_k}^{n}|+|Z_s^{n}-\bar{Z}_{s}^{n}|+|\Gamma_s^{n}-\bar{\Gamma}_{s}^{n}|\right)\upd s\right]\\
&+ \mathbb{E}\left[\int_t^{t_{k+1}}\delta^nY_s\ |\delta^n Y_s|^{p-2}|\bar{Z}_{t_k}^n-\bar{Z}_{t_k}^{n,\pi}|\upd s\right]+ \mathbb{E}\left[\int_t^{t_{k+1}}\delta^nY_s\ |\delta^n Y_s|^{p-2}|\bar{\Gamma}_{t_k}^n-\bar{\Gamma}_{t_k}^{n,\pi}|\upd s\right].
\end{align}
We repeatedly use the inequality $ab\le \alpha a^2+b^2/\alpha$ to get
\begin{align}
&\ \mathbb{E}\left[|\delta^n Y_t|^p\right]+\kappa_p \mathbb{E}\left[\int_t^{t_{k+1}}|\delta^nY_s|^{p-2}|\delta^nZ_s|^2\upd s\right]+ \kappa_p\mathbb{E}\left[\int_t^{t_{k+1}}\int_{\mathbb{R}}|\delta^nY_s|^{p-2}|\delta^nU_s(e)e|^2\nu^n(\upd e)\upd s\right]\\
\le&\ \mathbb{E}\left[|\delta^n Y_{t_{k+1}}|^p\right]+ (\alpha+\beta+\gamma)\mathbb{E}\left[\int_t^{t_{k+1}}|\delta^nY_s|^p\upd s\right]\\
& +\ \frac{C_p}{\alpha}\mathbb{E}\left[\int_t^{t_{k+1}} |\delta^n Y_s|^{p-2}\left(N^{-1}+|X_s^{n}-X_{t_k}^{n,\pi}|^2+|\delta^n Y_{t_k}|^2+|Y_s^{n}-Y_{t_k}^{n}|^2+|Z_s^{n}-\bar{Z}_{s}^{n}|^2+|\Gamma_s^{n}-\bar{\Gamma}_{s}^{n}|^2\right)\upd s\right]\\
&+ \frac{1}{\beta}\mathbb{E}\left[\int_t^{t_{k+1}} |\delta^n Y_s|^{p-2}|\bar{Z}_{t_k}^n-\bar{Z}_{t_k}^{n,\pi}|^2\upd s\right]+ \frac{1}{\gamma}\mathbb{E}\left[\int_t^{t_{k+1}} |\delta^n Y_s|^{p-2}|\bar{\Gamma}_{t_k}^n-\bar{\Gamma}_{t_k}^{n,\pi}|^2\upd s\right].
\end{align}
Next we apply Young's inequality
\begin{align}
&\ \mathbb{E}\left[|\delta^n Y_t|^p\right]+\kappa_p \mathbb{E}\left[\int_t^{t_{k+1}}|\delta^nY_s|^{p-2}|\delta^nZ_s|^2\upd s\right]+ \kappa_p\mathbb{E}\left[\int_t^{t_{k+1}}\int_{\mathbb{R}}|\delta^nY_s|^{p-2}|\delta^nU_s(e)e|^2\nu^n(\upd e)\upd s\right]\\
\le&\ \mathbb{E}\left[|\delta^n Y_{t_{k+1}}|^p\right]+ C_p \left(\alpha+\beta+\gamma+\frac{1}{\alpha}\right)\mathbb{E}\left[\int_t^{t_{k+1}}|\delta^nY_s|^p\upd s\right]\\
& +\ \frac{C_p}{\alpha}\mathbb{E}\left[\int_t^{t_{k+1}}\left(N^{-p/2}+|X_s^{n}-X_{t_k}^{n,\pi}|^p+|\delta^n Y_{t_k}|^p+|Y_s^{n}-Y_{t_k}^{n}|^p+|Z_s^{n}-\bar{Z}_{s}^{n}|^p+|\Gamma_s^{n}-\bar{\Gamma}_{s}^{n}|^p\right)\upd s\right]\\
&+ \frac{1}{\beta}\mathbb{E}\left[\int_t^{t_{k+1}} |\delta^n Y_s|^{p-2}|\bar{Z}_{t_k}^n-\bar{Z}_{t_k}^{n,\pi}|^2\upd s\right]+ \frac{1}{\gamma}\mathbb{E}\left[\int_t^{t_{k+1}} |\delta^n Y_s|^{p-2}|\bar{\Gamma}_{t_k}^n-\bar{\Gamma}_{t_k}^{n,\pi}|^2\upd s\right].
\end{align}
Because we know from the second terms in the chain of inequalities \eqref{eq:ZbarHp}, we have 
\begin{equation}
\mathbb{E}\left[\int_t^{t_{k+1}}|\bar{Z}_{t_k}^n-\bar{Z}_{t_k}^{n,\pi}|^2\upd s\right]\le C_2 \mathbb{E}\left[\int_{t_k}^{t_{k+1}}|\delta^nZ_s|^2\upd s\right]
\end{equation}
and analogously
\begin{equation}
\mathbb{E}\left[\int_t^{t_{k+1}}|\bar{\Gamma}_{t_k}^n-\bar{\Gamma}_{t_k}^{n,\pi}|^2\upd s\right]\le C_2 \mathbb{E}\left[\int_{t_k}^{t_{k+1}}|\delta^nU_s|^2\upd s\right],
\end{equation}
for a constant $C_2>0$, we can choose $\beta,\gamma>0$ independent of $N$ such that
\begin{equation}
\kappa_p \mathbb{E}\left[\int_t^{t_{k+1}}|\delta^nY_s|^{p-2}|\delta^nZ_s|^2\upd s\right]\ge\frac{1}{\beta}\mathbb{E}\left[\int_t^{t_{k+1}} |\delta^n Y_s|^{p-2}|\bar{Z}_{t_k}^n-\bar{Z}_{t_k}^{n,\pi}|^2\upd s\right]
\end{equation}
and
\begin{equation}
\kappa_p\mathbb{E}\left[\int_t^{t_{k+1}}\int_{\mathbb{R}}|\delta^nY_s|^{p-2}|\delta^nU_s(e)e|^2\nu^n(\upd e)\upd s\right]\ge \frac{1}{\gamma}\mathbb{E}\left[\int_t^{t_{k+1}} |\delta^n Y_s|^{p-2}|\bar{\Gamma}_{t_k}^n-\bar{\Gamma}_{t_k}^{n,\pi}|^2\upd s\right].
\end{equation}
This and \eqref{eq:bdXdiscr} imply that
\begin{align}
\mathbb{E}\left[|\delta^n Y_t|^p\right]
\le&\ \mathbb{E}\left[|\delta^n Y_{t_{k+1}}|^p\right]+ C_p \left(\alpha+\frac{1}{\alpha}\right)\mathbb{E}\left[\int_t^{t_{k+1}}|\delta^nY_s|^p\upd s\right]\\
& +\ \frac{C_p}{\alpha}\int_t^{t_{k+1}}\mathbb{E}\left[N^{-p/2}+N^{-p}\mathfrak{m}^1(n)^{p}+|\delta^n Y_{t_k}|^p+|Y_s^{n}-Y_{t_k}^{n}|^p+|Z_s^{n}-\bar{Z}_{s}^{n}|^p+|\Gamma_s^{n}-\bar{\Gamma}_{s}^{n}|^p\right]\upd s,
\end{align}
for $t\in[t_k,t_{k+1})$ and thus
\begin{align}
\mathbb{E}\left[|\delta^n Y_t|^p\right]
\le&\ \mathbb{E}\left[|\delta^n Y_{t_{k+1}}|^p\right]+ C_p \left(\alpha+\frac{1}{\alpha}\right)\mathbb{E}\left[\int_t^{t_{k+1}}|\delta^nY_s|^p\upd s\right]\\
& +\ \frac{C_p}{\alpha}\left((t_{k+1}-t)\left(N^{-p/2}+N^{-p}\mathfrak{m}^1(n)^{p}+\mathbb{E}\left[|\delta^n Y_{t_k}|^p\right]\right)+\bar{B}_k\right),
\end{align}
where
\begin{equation}
\bar{B}_k:= \int_{t_k}^{t_{k+1}}\left(\mathbb{E}\left[|Y_s^{n}-Y_{t_k}^{n}|^p\right]+\mathbb{E}\left[|Z_s^{n}-\bar{Z}_{s}^{n}|^p\right]+\mathbb{E}\left[|\Gamma_s^{n}-\bar{\Gamma}_{s}^{n}|^p\right]\right)\upd s.
\end{equation}
Using Gronwall's Lemma, we can choose $\alpha$ independent of $N$ such that
\begin{equation}\label{eq:recurEuler1}
\mathbb{E}\left[|\delta^n Y_t|^p\right]\le \mathbb{E}\left[|\delta^n Y_{t_{k+1}}|^p\right]+C_p\left((t_{k+1}-t)\left(N^{-p/2}+N^{-p}\mathfrak{m}^1(n)^{p}+\mathbb{E}\left[|\delta^n Y_{t_k}|^p\right]\right)+\bar{B}_k\right).
\end{equation}
If we take $t=t_k$ in \eqref{eq:recurEuler1} we get
\begin{equation}\label{eq:recurEuler2}
\mathbb{E}\left[|\delta^n Y_{t_k}|^p\right]\le \mathbb{E}\left[|\delta^n Y_{t_{k+1}}|^p\right]+C_p\left((t_{k+1}-t_k)\left(N^{-p/2}+N^{-p}\mathfrak{m}^1(n)^{p}+\mathbb{E}\left[|\delta^n Y_{t_k}|^p\right]\right)+\bar{B}_k\right).
\end{equation}
Plugging \eqref{eq:recurEuler2} into \eqref{eq:recurEuler1} iteratively, combined with the Lipschitz condition for the terminal value $g(X_T^n)-g(X_T^{n,\pi})$ and the bound \eqref{eq:forwardEulerError} we obtain
\begin{equation}
\mathbb{E}\left[|\delta^n Y_{t}|^p\right]\le C_p \left(N^{-p/2}+N^{-p}\mathfrak{m}^1(n)^{p}+\bar{B}\right),
\end{equation}
for $t\in[0,T]$, where
\begin{equation}
\bar{B}:= \sum_{k=0}^{N+J^n-1}\bar{B}_k.
\end{equation}
We can take the supremum over all $t$ and conclude
\begin{equation}\label{eq:supdeltanY}
\sup_{0\le t\le T}\mathbb{E}\left[|\delta^n Y_{t}|^p\right]\le C_p \left(N^{-p/2}+N^{-p}\mathfrak{m}^1(n)^{p}+\bar{B}\right).
\end{equation}

\noindent \textbf{Step 2:}
We also can show that \eqref{eq:recurEuler1} holds for taking the supremum over $[t_k,t_{k+1})$, i.e.,
\begin{equation}\label{eq:recurEuler3}
\mathbb{E}\left[\sup_{t_k\le t<t_{k+1}}|\delta^n Y_t|^p\right]\le \mathbb{E}\left[|\delta^n Y_{t_{k+1}}|^p\right]+C_p\left((t_{k+1}-t)\left(N^{-p/2}+N^{-p}\mathfrak{m}^1(n)^{p}+\mathbb{E}\left[|\delta^n Y_{t_k}|^p\right]\right)+\bar{B}_k\right).
\end{equation}
This follows like in Step 1 by using Itô's formula (without the expectations)
\begin{align}
&\ |\delta^n Y_t|^p+\kappa_p \int_t^{t_{k+1}}|\delta^nY_s|^{p-2}|\delta^nZ_s|^2\upd s+ \kappa_p\int_t^{t_{k+1}}\int_{\mathbb{R}}|\delta^nY_s|^{p-2}|\delta^nU_s(e)e|^2\nu^n(\upd e)\upd s\\
\le&\ |\delta^n Y_{t_{k+1}}|^p+ C_p \left(\alpha+\beta+\gamma+\frac{1}{\alpha}\right)\int_t^{t_{k+1}}|\delta^nY_s|^p\upd s\\
& +\ \frac{C_p}{\alpha}\int_t^{t_{k+1}}\left(N^{-p/2}+|X_s^{n}-X_{t_k}^{n,\pi}|^p+|\delta^n Y_{t_k}|^p+|Y_s^{n}-Y_{t_k}^{n}|^p+|Z_s^{n}-\bar{Z}_{s}^{n}|^p+|\Gamma_s^{n}-\bar{\Gamma}_{s}^{n}|^p\right)\upd s\\
&+ \frac{1}{\beta}\int_t^{t_{k+1}} |\delta^n Y_s|^{p-2}|\bar{Z}_{t_k}^n-\bar{Z}_{t_k}^{n,\pi}|^2\upd s+ \frac{1}{\gamma}\int_t^{t_{k+1}} |\delta^n Y_s|^{p-2}|\bar{\Gamma}_{t_k}^n-\bar{\Gamma}_{t_k}^{n,\pi}|^2\upd s + M_t,\label{eq:mainproof8}
\end{align}
where 
\begin{equation}
M_t= \int_t^{t_{k+1}} \delta^n Y_{s-} |\delta^n Y_{s-}|^{p-2}\delta^n Z_s\upd B_s + \int_t^{t_{k+1}}\int_{\mathbb{R}}\delta^n Y_{s-}|\delta^n Y_{s-}|^{p-2}\delta^n U_s(e)e\widetilde{\mu^n}(\upd e,\upd s)
\end{equation}
denotes the martingales which can be handled with the Burkholder-Davis-Gundy inequality:
\begin{align}
&\mathbb{E}\left[\sup_{t_k\le t< t_{k+1}}\left|\int_t^{t_{k+1}} \delta^n Y_{s-} |\delta^n Y_{s-}|^{p-2}\delta^n Z_s\upd B_s\right|\right] \le C_p \mathbb{E}\left[\left(\int_{t_k}^{t_{k+1}}|\delta^n Y_{s}|^{2p-2}|\delta^n Z_s|^2\upd s\right)^{1/2}\right]\\
\le& \frac{1}{4p}\mathbb{E}\left[\sup_{t_k\le t< t_{k+1}}|\delta^n Y_t|^p\right]+pC_p^2 \mathbb{E}\left[\int_{t_k}^{t_{k+1}}|\delta^n Y_s|^{p-2}|\delta^n Z_s|^2\upd s\right]
\end{align}
and
\begin{align}
&\mathbb{E}\left[\sup_{t_k\le t< t_{k+1}}\left|\int_t^{t_{k+1}}\int_{\mathbb{R}}\delta^n Y_{s-}|\delta^n Y_{s-}|^{p-2}\delta^n U_s(e)e\widetilde{\mu^n}(\upd e,\upd s)\right|\right]\\
\le &C_p \mathbb{E}\left[\left(\int_{t_k}^{t_{k+1}}\int_{\mathbb{R}}|\delta^n Y_{s-}|^{2p-2}\delta^n U_s(e)^2e^2\mu^n(\upd e,\upd s)\right)^{1/2}\right]\\
\le &\frac{1}{4p}\mathbb{E}\left[\sup_{t_k\le t< t_{k+1}}|\delta^n Y_t|^p\right]+pC_p^2\mathbb{E}\left[\int_{t_k}^{t_{k+1}}\int_{\mathbb{R}}|\delta^n Y_s|^{p-2}|\delta^n U_s(e)|^2e^2 \nu^n(\upd e)\upd s\right].
\end{align}
Taking the supremum and expectations of \eqref{eq:mainproof8}, using the above two bounds and proceeding as in Step 1 yields \eqref{eq:recurEuler3}.

\noindent \textbf{Step 3:} 
The next step controls 
\begin{equation}
\mathbb{E}\left[\left(\int_{t_k}^{t_{k+1}}|\delta^nZ_s|^2\upd s\right)^{p/2}+\left(\int_{t_k}^{t_{k+1}}\int_{\mathbb{R}}|\delta^nU_s(e)|^2e^2\nu^n(\upd e)\upd s\right)^{p/2}\right].
\end{equation}
We start by applying Itô's formula to $|\delta^n Y_t|^2$ on $[t_k,t_{k+1})$
\begin{align}
& |\delta^nY_t|^2+ \int_t^{t_{k+1}}\delta^n Z_s^2\upd s+\int_t^{t_{k+1}}\int_{\mathbb{R}}\delta^n U_s(e)^2e^2\mu^n(\upd e,\upd s)\\
=&\ |\delta^n Y_{t_{k+1}}|^2+2\int_t^{t_{k+1}}\delta^nY_s\ \delta^nf(\Theta_s) \upd s\\
& \ -2\int_{t}^{t_{k+1}}\delta^n Y_{s-}\ \delta^n Z_s\upd B_s - 2\int_{t}^{t_{k+1}}\int_{\mathbb{R}}\delta^n Y_{s-}\ \delta^nU_s(e)e\widetilde{\mu^n}(\upd e,\upd s).
\end{align}
The Lipschitz and Hölder condition on $f$ then imply
\begin{align}
&  |\delta^nY_t|^2+\int_t^{t_{k+1}}\delta^n Z_s^2\upd s+\int_t^{t_{k+1}}\int_{\mathbb{R}}\delta^n U_s(e)^2e^2\mu^n(\upd e,\upd s)\\
\le&\ |\delta^n Y_{t_{k+1}}|^2+2\int_t^{t_{k+1}}\delta^nY_s\left((T/N)^{-1/2}+|X_s^n-X_s^{n,\pi}|+|Y_s^n-\bar{Y}_{t_k}^{n,\pi}|+|Z_s^n-\bar{Z}_{t_k}^{n,\pi}|+|\Gamma_s^n-\bar{\Gamma}_{t_k}^{n,\pi}|\right) \upd s\\
& \ -2\int_{t}^{t_{k+1}}\delta^n Y_{s-}\ \delta^n Z_s\upd B_s - 2\int_{t}^{t_{k+1}}\int_{\mathbb{R}}\delta^n Y_{s-}\ \delta^nU_s(e)e\widetilde{\mu^n}(\upd e,\upd s).
\end{align}
Again we use the inequality $ab\le \alpha a^2+b^2/\alpha$ to get
\begin{align}
&  |\delta^nY_t|^2+\int_t^{t_{k+1}}\delta^n Z_s^2\upd s+\int_t^{t_{k+1}}\int_{\mathbb{R}}\delta^n U_s(e)^2e^2\mu^n(\upd e,\upd s)\\
\le&\ |\delta^n Y_{t_{k+1}}|^2+(\alpha+\beta+\gamma)\int_t^{t_{k+1}}|\delta^nY_s|^2\upd s\\
&\ + \frac{C_p}{\alpha}\int_t^{t_{k+1}}\left(N^{-1}+|X_s^n-X_s^{n,\pi}|^2+|\delta^nY_{t_k}|^2+|\bar{Z}_{t_k}^{n}-\bar{Z}_{t_k}^{n,\pi}|^2+|\bar{\Gamma}_{t_k}^{n}-\bar{\Gamma}_{t_k}^{n,\pi}|^2\right) \upd s\\
&\ +\frac{C_p}{\alpha}\int_t^{t_{k+1}}|Y_s^n-Y_{t_k}^{n}|^2\upd s +   \frac{C_p}{\beta}\int_t^{t_{k+1}}|Z_s^n-\bar{Z}_{s}^{n}|^2 \upd s+   \frac{C_p}{\gamma}\int_t^{t_{k+1}}|\Gamma_s^n-\bar{\Gamma}_{s}^{n}|^2 \upd s\\
& \  -2\int_{t}^{t_{k+1}}\delta^n Y_{s-}\ \delta^n Z_s\upd B_s - 2\int_{t}^{t_{k+1}}\int_{\mathbb{R}}\delta^n Y_{s-}\ \delta^nU_s(e)e\widetilde{\mu^n}(\upd e,\upd s).
\end{align}
Next we take powers
\begin{align}
&  |\delta^nY_t|^p+\left(\int_t^{t_{k+1}}\delta^n Z_s^2\upd s\right)^{p/2} + \left(\int_t^{t_{k+1}}\int_{\mathbb{R}}\delta^n U_s(e)^2e^2\mu^n(\upd e,\upd s)\right)^{p/2}\\
\le&\ C_p |\delta^n Y_{t_{k+1}}|^p + C_p(\alpha+\beta+\gamma)^{p/2}\left|\int_t^{t_{k+1}}|\delta^nY_s|^2\upd s\right|^{p/2}+ \frac{C_p}{\alpha^{p/2}}N^{-p} + \frac{C_p}{\alpha^{p/2}}\left|\int_t^{t_{k+1}}|X_s^n-X_s^{n,\pi}|^2\upd s\right|^{p/2}\\
&\  + \frac{C_p}{\alpha^{p/2}}N^{-p/2}|\delta^nY_{t_k}|^p+\frac{C_p}{\beta^{p/2}}(t_{k+1}-t_k)^{p/2}|\bar{Z}_{t_k}^{n}-\bar{Z}_{t_k}^{n,\pi}|^p+\frac{C_p}{\gamma^{p/2}}(t_{k+1}-t_k)^{p/2}|\bar{\Gamma}_{t_k}^{n}-\bar{\Gamma}_{t_k}^{n,\pi}|^p \\
&\ +\frac{C_p}{\alpha}\left|\int_t^{t_{k+1}}\left(|Y_s^n-Y_{t_k}^{n}|^2+|Z_s^n-\bar{Z}_{s}^{n}|^2+|\Gamma_s^n-\bar{\Gamma}_{s}^{n}|^2\right) \upd s\right|^{p/2}\\
& \  +C_p\left|\int_{t}^{t_{k+1}}\delta^n Y_{s-}\ \delta^n Z_s\upd B_s\right|^{p/2} + C_p\left|\int_{t}^{t_{k+1}}\int_{\mathbb{R}}\delta^n Y_{s-}\ \delta^nU_s(e)e\widetilde{\mu^n}(\upd e,\upd s)\right|^{p/2},
\end{align}
and expectations to get
\begin{align}
&  \mathbb{E}\left[|\delta^nY_t|^p\right]+\mathbb{E}\left[\left(\int_t^{t_{k+1}}\delta^n Z_s^2\upd s\right)^{p/2} + \left(\int_t^{t_{k+1}}\int_{\mathbb{R}}\delta^n U_s(e)^2e^2\mu^n(\upd e,\upd s)\right)^{p/2}\right]\\
\le&\ C_p \mathbb{E}\left[|\delta^n Y_{t_{k+1}}|^p\right] + C_p(\alpha+\beta+\gamma)^{p/2}\mathbb{E}\left[\left|\int_t^{t_{k+1}}|\delta^nY_s|^2\upd s\right|^{p/2}\right]+ \frac{C_p}{\alpha^{p/2}}N^{-p}\\
&\ + \frac{C_p}{\alpha^{p/2}}\mathbb{E}\left[\left|\int_t^{t_{k+1}}|X_s^n-X_s^{n,\pi}|^2\upd s\right|^{p/2}\right]+ \frac{C_p}{\alpha^{p/2}}N^{-p/2}\mathbb{E}\left[|\delta^nY_{t_k}|^p\right]\\
&\  + \frac{C_p}{\beta^{p/2}}(t_{k+1}-t_k)^{p/2}\mathbb{E}\left[|\bar{Z}_{t_k}^{n}-\bar{Z}_{t_k}^{n,\pi}|^p\right]+\frac{C_p}{\gamma^{p/2}}(t_{k+1}-t_k)^{p/2}\mathbb{E}\left[|\bar{\Gamma}_{t_k}^{n}-\bar{\Gamma}_{t_k}^{n,\pi}|^p\right] \\
&\ +\frac{C_p}{\alpha}\mathbb{E}\left[\left|\int_{t_k}^{t_{k+1}}\left(|Y_s^n-Y_{t_k}^{n}|^2+|Z_s^n-\bar{Z}_{s}^{n}|^2+|\Gamma_s^n-\bar{\Gamma}_{s}^{n}|^2\right) \upd s\right|^{p/2}\right]\\
& \  +C_p\mathbb{E}\left[\left|\int_{t}^{t_{k+1}}\delta^n Y_{s-}\ \delta^n Z_s\upd B_s\right|^{p/2}\right] + C_p\mathbb{E}\left[\left|\int_{t}^{t_{k+1}}\int_{\mathbb{R}}\delta^n Y_{s-}\ \delta^nU_s(e)e\widetilde{\mu^n}(\upd e,\upd s)\right|^{p/2}\right].\label{eq:mainproof7}
\end{align}
We discuss the terms in \eqref{eq:mainproof7} separately. First, we recall that by \eqref{eq:recurEuler2} $\mathbb{E}\left[|\delta^n Y_{t_{k}}|^p\right]\le \mathbb{E}\left[|\delta^n Y_{t_{k+1}}|^p\right]+C_p\left((t_{k+1}-t_k)\left(N^{-p/2}+N^{-p}\mathfrak{m}^1(n)^{p}+\mathbb{E}\left[|\delta^n Y_{t_k}|^p\right]\right)+\bar{B}_k\right)$ and, by \eqref{eq:recurEuler3} and additionally invoking Jensen's inequality,
\begin{align}
&\mathbb{E}\left[\left|\int_t^{t_{k+1}}|\delta^nY_s|^2\upd s\right|^{p/2}\right]\le C_p N^{-p/2}\mathbb{E}\left[\sup_{t_k\le s< t_{k+1}}|\delta^n Y_{s}|^p\right]\\
\le&\ C_p \left(N^{-p/2}\mathbb{E}\left[|\delta^n Y_{t_{k+1}}|^p\right]+N^{-p-1}+N^{-p-p/2-1}\mathfrak{m}^1(^n)^{p}+N^{-p/2-1}\mathbb{E}\left[|\delta^n Y_{t_k}|^p\right]\right).
\end{align}
Second, in a similar manner, the term with the forward SDE $X$ is bounded by $C_p\left( N^{-p}+N^{-p-p/2}\mathfrak{m}^1(n)^{p}\right)$. Third, recalling \eqref{eq:ZbarHp} and \eqref{eq:GammabarHp} we note that
\begin{equation}
(t_{k+1}-t_k)^{p/2}\mathbb{E}\left[|\bar{Z}_{t_k}^{n}-\bar{Z}_{t_k}^{n,\pi}|^p\right]\le C_p ||\delta^nZ||_{\mathbb{H}_{[t_k,t_{k+1}]}^p}^p
\end{equation}
and
\begin{equation}
(t_{k+1}-t_k)^{p/2}\mathbb{E}\left[|\bar{\Gamma}_{t_k}^{n}-\bar{\Gamma}_{t_k}^{n,\pi}|^p\right]\le C_p ||\delta^nU||_{\mathbb{L}_{\mu^n,[t_k,t_{k+1}]}^p}^p.
\end{equation}
Fourth, by Jensen's inequality
\begin{equation}
\mathbb{E}\left[\left|\int_{t_k}^{t_{k+1}}\left(|Y_s^n-Y_{t_k}^{n}|^2+|Z_s^n-\bar{Z}_{s}^{n}|^2+|\Gamma_s^n-\bar{\Gamma}_{s}^{n}|^2\right) \upd s\right|^{p/2}\right]\le C_p N^{-p/2+1}\bar{B}_k.
\end{equation}
Finally, we can apply the Burkholder-Davis-Gundy inequality and Young's inequality to the martingales in \eqref{eq:mainproof7}:
\begin{align}
&C_p\mathbb{E}\left[\left|\int_{t}^{t_{k+1}}\delta^n Y_{s-}\ \delta^n Z_s\upd B_s\right|^{p/2}\right]\\
\le &\ C_p \mathbb{E}\left[\left(\int_{t_k}^{t_{k+1}}|\delta^n Y_{s-}|^2\ |\delta^n Z_s|^2\upd s\right)^{p/4}\right]\\
\le&\ \frac{C_p^2}{4}\mathbb{E}\left[\sup_{t_k\le s<t_{k+1}} |\delta^nY_s|^p\right]+\frac{1}{2}\mathbb{E}\left[\left(\int_{t_k}^{t_{k+1}}|\delta^n Z_s|^2\upd s\right)^{p/2}\right]\\
\le&\ C_p\mathbb{E}\left[|\delta^n Y_{t_{k+1}}|^p\right]+C_p\left((t_{k+1}-t)\left(N^{-p/2}+N^{-p}\mathfrak{m}^1(n)^{p}+\mathbb{E}\left[|\delta^n Y_{t_k}|^p\right]\right)+\bar{B}_k\right)\\
&\ \ \ +\frac{1}{2}\mathbb{E}\left[\left(\int_{t_k}^{t_{k+1}}|\delta^n Z_s|^2\upd s\right)^{p/2}\right],
\end{align}
where the last inequality follows by \eqref{eq:recurEuler3}. Analogously,
\begin{align}
&C_p\mathbb{E}\left[\left|\int_{t}^{t_{k+1}}\int_{\mathbb{R}}\delta^n Y_{s-}\ \delta^nU_s(e)e\widetilde{\mu^n}(\upd e,\upd s)\right|^{p/2}\right]\\
\le&\ C_p \mathbb{E}\left[\left(\int_{t_k}^{t_{k+1}}|\delta^n Y_{s-}|^2\ |\delta^n U_s(e)e|^2\mu^n(\upd e,\upd s)\right)^{p/4}\right]\\
\le&\ \frac{C_p^2}{4}\mathbb{E}\left[\sup_{t_k\le s<t_{k+1}} |\delta^nY_s|^p\right]+\frac{1}{2}\mathbb{E}\left[\left(\int_{t_k}^{t_{k+1}}|\delta^n U_s(e)e|^2\mu^n(\upd e,\upd s)\right)^{p/2}\right]\\
\le&\ C_p\mathbb{E}\left[|\delta^n Y_{t_{k+1}}|^p\right]+C_p\left((t_{k+1}-t)\left(N^{-p/2}+N^{-p}\mathfrak{m}^1(n)^{p}+\mathbb{E}\left[|\delta^n Y_{t_k}|^p\right]\right)+\bar{B}_k\right)\\
&\ \ \ +\frac{1}{2}\mathbb{E}\left[\left(\int_{t_k}^{t_{k+1}}|\delta^n U_s(e)e|^2\mu^n(\upd e,\upd s)\right)^{p/2}\right].
\end{align}
We again use
\begin{equation}
\mathbb{E}\left[\left(\int_{t_k}^{t_{k+1}}|\delta^n U_s(e)e|^2\nu^n(\upd e)\upd s\right)^{p/2}\right]\le d_p\mathbb{E}\left[ \left(\int_{t_k}^{t_{k+1}}|\delta^n U_s(e)e|^2\mu^n(\upd e,\upd s)\right)^{p/2}\right]
\end{equation}
and conclude that, for $t=t_k$, we can choose constants $\alpha$, $\beta$ and $\gamma$ independent of $N$ such that \eqref{eq:mainproof7} can be simplified to
\begin{align}
&  \mathbb{E}\left[|\delta^nY_{t_k}|^p\right]+\mathbb{E}\left[\left(\int_{t_k}^{t_{k+1}}\delta^n Z_s^2\upd s\right)^{p/2} + \left(\int_{t_k}^{t_{k+1}}\int_{\mathbb{R}}\delta^n U_s(e)^2e^2\mu^n(\upd e,\upd s)\right)^{p/2}\right]\\
\le&\ C_p \mathbb{E}\left[|\delta^n Y_{t_{k+1}}|^p\right]+C_p\left((t_{k+1}-t_k)\left(N^{-p/2}+N^{-p}\mathfrak{m}^1(n)^{p}+\mathbb{E}\left[|\delta^n Y_{t_k}|^p\right]\right)+\bar{B}_k\right).\label{eq:mainproof9}
\end{align}
Now we can sum up equation \eqref{eq:mainproof9}. Together with \eqref{eq:recurEuler2} and the Lipschitz condition for the terminal value we obtain
\begin{equation}
\mathbb{E}\left[\left(\int_{0}^{T}\delta^n Z_s^2\upd s\right)^{p/2} + \left(\int_{0}^{T}\int_{\mathbb{R}}\delta^n U_s(e)^2e^2\mu^n(\upd e,\upd s)\right)^{p/2}\right] \le C_p\left(N^{-p/2}+N^{-p}\mathfrak{m}^1(n)^{p}+\bar{B}\right).
\end{equation}
Joining Step 1 with Step 3 then implies that
\begin{equation}
\label{eq:EulerErrBbar}
Err_{\pi,p}(Y^n,Z^n,U^n)^p\le C_p \left( N^{-p/2}+N^{-p}\mathfrak{m}^1(n)^{p}+\bar{B}\right).
\end{equation}

\noindent \textbf{Step 4:} 
It remains to show that $\bar{B}\le C_pN^{-p/2}$. For the first term in $\bar{B}$, we recall that $Y^n$ solves \eqref{eq:nYSDEv3} and hence
\begin{equation}
\mathbb{E}\left[|Y_t^n-Y_{t_k}^n|^p\right]\le C_p\int_{t_k}^t\mathbb{E}\left[|f(s,X_s^n,Y_s^n,Z_s^n,\Gamma_s^n)|^p+|Z_s^n|^p+\int_{\mathbb{R}}|U_s^n(e)e|^p\nu(\upd e)\right]\upd s.
\end{equation}
The Lipschitz property of $f$ combined with \eqref{eq:standardestimate} implies
\begin{equation}
\sum_{k=0}^{N+J^n-1}\int_{t_k}^{t_{k+1}}\mathbb{E}\left[|Y_t^n-Y_{t_k}^n|^p\right]\upd t \le C_pN^{-p/2}.
\end{equation} 

For the second and third term of $\bar{B}$ we use Assumption \ref{ass:invofh}. We follow the proofs of \cite{Bouchard2008} (who used a multivariate version of Assumption \ref{ass:invofh}). \citeauthor{Bouchard2008} (\citeyear[Propositions 4.5-4.6 \& Theorem 2.1]{Bouchard2008}) proved that the regularities of $Z^n$ and $\Gamma^n$ are bounded by $C_2 N^{-1}$ for $p=2$. Replacing $p=2$ with a $p\ge2$ is a straightforward extension of their proofs. \cite{Zhang2004} and  proved this is independent of the specific partition only depending on its mesh $1/N$. This implies $\bar{B}\le C_pN^{-p/2}$ and finally the statement follows by joining Steps 1-4.
\end{proof}

\begin{rem}\label{rem:Ass2}
Following the argument of \cite{Bouchard2008}, if Assumption \ref{ass:invofh}, which is a one-dimensional special case of their Assumption \textbf{H}, does not hold the error bound \eqref{eq:EulerErr} is not valid anymore. This is because the regularity in $Z$, i.e., $||Z^n-\bar{Z}^n||_{\mathbb{H}^p}^p$ is not bounded by $C_pN^{-p/2}$ in this case. However, one can show without using Assumption \ref{ass:invofh} that, for any $\varepsilon>0$, there exists a constant $C_{p,\varepsilon}$ such that
\begin{equation}
||Z^n-\bar{Z}^n||_{\mathbb{H}^p}^p\le C_{p,\varepsilon} N^{-p/2+\varepsilon}.
\end{equation}
Note that the regularities of $Y^n$ and $\Gamma^n$ remain unaffected whether Assumption \ref{ass:invofh} is fulfilled or not, i.e., $||Y^n-\bar{Y}^n||_{\mathcal{S}^p}^p\le C_pN^{-p/2}$ and $||\Gamma^n-\bar{\Gamma}^n||_{\mathbb{H}^p}^p\le C_pN^{-p/2}$ even without Assumption \ref{ass:invofh}. Furthermore, if either $a\equiv0$, or the generator $f$ is independent of $Z$ Theorem \ref{thm:EulerErr} holds without Assumption \ref{ass:invofh}, see \cite{Elie2006} for a discussion.
\end{rem}

\begin{rem}
Theorem \ref{thm:EulerErr} can equivalently formulated as that under the above assumptions there exists a constant $C_p$ independent of $n$ and $\pi$
such that
\begin{equation}
Err_{\pi,p}(Y^n,Z^n,U^n)\le C_p \mathbb{E}\left[\sup_{t\in[0,T]}|X_t^n-X_t^{n,\pi}|^p\right]^{1/p}.
\end{equation}
\end{rem}

\begin{rem}
Instead of the implicit scheme \eqref{eq:EulerSchemeY}, one could use an explicit scheme where we replace $\bar{Y}_{t_k}^{n,\pi}$ by $\bar{Y}_{t_{k+1}}^{n,\pi}$ in the argument of $h$. The advantage is that we do not need a fixed-point procedure in this case. One disadvantage is that the conditional expectations are more difficult to estimate. We refer to \cite{Bouchard2008} and \cite{Elie2006} for details.
\end{rem}


Using Theorems \ref{thm:approxY} and \ref{thm:EulerErr}, we deduce a bound for the approximation-discretization error between the original backward SDE \eqref{eq:YSDEv3} and the scheme \eqref{eq:EulerSchemeY} which is defined as
\begin{align}
Err_{n,\pi,p}(Y,Z,U):=&\ \left(\sup_{0\le t\le T} \mathbb{E}\left[|Y_t-\bar{Y}_t^{n,\pi}|^p\right]+||Z-\bar{Z}^{n,\pi}||_{\mathbb{H}^p}^p+||\Gamma-\bar{\Gamma}^{n,\pi}||_{\mathbb{H}^p}^p\right)^{1/p}.
\end{align}
The approximation-discretization error for the forward SDE
\begin{equation}\label{eq:forwardError}
\max_{k<N}\mathbb{E}\left[\sup_{t\in[t_k,t_{k+1}]}|X_t-X_t^{n,\pi}|^p\right]\le C_p \left(n^{-p/2}+\sigma^p(n)+\sigma^2(n)^{p/2}+N^{-p}\mathfrak{m}^1(n)^{p}\right),
\end{equation} 
is straightforward combining \eqref{eq:approxerrorX} with \eqref{eq:forwardEulerError}.

\begin{cor}\label{cor:Error}
Under Assumptions \ref{ass:Lipschitz} and \ref{ass:invofh}, the approximation-discretization error is bounded by
\begin{equation}
\label{eq:Error}
Err_{n,\pi,p}(Y,Z,U)\le C_p \left(N^{-1/2}+\sigma^p(n)^{1/p}+\sigma^2(n)^{1/2}+N^{-1}\mathfrak{m}^1(n)\right).
\end{equation}
\end{cor}

\begin{proof}
This is an easy consequence because
\begin{align}
Err_{n,\pi,p}(Y,Z,U)^p \le&\ C_p \Biggr( \sup_{0\le t\le T} \mathbb{E}\left[|Y_t-Y_t^n|^p+|Y_t^n-\bar{Y}_t^{n,\pi}|^p\right]+||Z-Z^n||_{\mathbb{H}^p}^p+||Z^n-\bar{Z}^{n,\pi}||_{\mathbb{H}^p}^p\\
&\ \ \ \ \ +||\Gamma-\Gamma^{n}||_{\mathbb{H}^p}^p+||\Gamma^n-\bar{\Gamma}^{n,\pi}||_{\mathbb{H}^p}^p\Biggr).
\end{align}
Using Remark \ref{rem:boundonrho} we can show
\begin{equation}
\left(\int_{\mathbb{R}}\rho(e)(U_s(e)-U_s^n(e))e\nu^n(\upd e)\right)^2\le C_p \int_{\mathbb{R}}(U_s(e)-U_s^n(e))^2e^2\nu^n(\upd e)
\end{equation}
and
\begin{equation}
\left(\int_{\mathbb{R}}\rho(e)U_s(e)e\bar{\nu}^n(\upd e)\right)^2\le C_p \int_{\mathbb{R}}U_s(e)^2e^2\bar{\nu}^n(\upd e)
\end{equation}
which imply
\begin{equation}
||\Gamma-\Gamma^{n}||_{\mathbb{H}^p}^p\le C_p \mathbb{E}\left[\left(\int_0^T\int_{\mathbb{R}}(U_s(e)-U_s^n(e))^2e^2\nu^n(\upd e)\upd s\right)^{p/2}+\left(\int_0^T\int_{\mathbb{R}}U_s(e)^2e^2\bar{\nu}^n(\upd e)\upd s\right)^{p/2}\right],
\end{equation}
and thus the result follows.
\end{proof}

We end this paper with some remarks about implementation of the scheme in practice.

\begin{rem}
The proposed scheme is not fully implementable in practice. One key step is the computation of the conditional expectations in \eqref{eq:EulerSchemeY} which has to be performed numerically. There are several methods to estimate these. Among them there are nonparametric kernel regression \cite[]{Bouchard2004,Lemor2006}, Malliavin regression \cite[]{Bouchard2004}, quantization \cite[]{Bally2003} and some other approaches. We discuss the nonparametric regression approach in some more detail which works by simulating $1\le m\le M$ paths $X^{n,\pi,m}$ of $X^{n,\pi}$ and initialize $\bar{Y}_T^{n,\pi,m}=g(X_T^{n,\pi,m})$. Then we regress $\bar{Y}_{t_{k+1}}^{n,\pi,m}$ and $\bar{Y}_{t_{k+1}}^{n,\pi,m}\Delta B_{k+1}^m$ and $\bar{Y}_{t_{k+1}}^{n,\pi,m}\int_{\mathbb{R}}\rho(e)e\widetilde{\mu^{n,m}}(\upd e,(t_k,t_{k+1}])$ on $X_{t_k}^{n,\pi,m}$. Details are presented in \cite{Elie2006}.

To compute the $L^p$ error between the original backward SDE and the numerical backward SDE taking into account approximation of the jump process, discretization and estimation of conditional expectations we have to sum up the error of Corollary \ref{cor:Error}, the error of a localization procedure and the statistical error by the kernel regression. \cite{Elie2006} derived the $L^p$ error of the localization procedure. Furthermore, \cite{Elie2006} derived the statistical error which is in terms of the Euclidean norm on $\mathbb{R}^M$. Since all norms on $\mathbb{R}^M$ are equivalent it is not much work to deduce a bound for the error in terms of the $p$-norm. All in all, if we choose some other parameters in the algorithm large enough, we can conclude that the total error is of the order $N^{-1/2}+\sigma^p(n)^{1/p}+\sigma^2(n)^{1/2}+N^{-1}\mathfrak{m}^1(n)$ under Assumptions \ref{ass:Lipschitz} and \ref{ass:invofh}.
\end{rem}

\section*{Acknowledgements}
Financial support of the German Research Foundation (Deutsche Forschungsgemeinschaft, DFG) via the Collaborative Research Center ``Statistical modelling of nonlinear dynamic processes'' (SFB 823, Teilprojekt A4) and via the DFG project 455257011 is gratefully acknowledged.

The author thanks two anonymous referees for suggestions which helped to substantially improve this paper. The author is grateful to Alexandre Popier and Christoph Hanck for valuable comments and discussions. Full responsibility is taken for
all remaining errors.

\bibliography{bibliography}
\bibliographystyle{agsm}

\end{document}